\documentclass[review]{elsarticle}

\usepackage{lineno}

\usepackage{epsfig}
\usepackage{amsmath}
\usepackage{amsfonts}
\usepackage{amsthm}
\usepackage{amsbsy}
\usepackage{appendix}
\usepackage{epstopdf}
\usepackage[table]{xcolor}
\usepackage{subfloat}
\usepackage{float}
\usepackage[thinlines]{easytable}
\usepackage{array}
\usepackage{color}
\usepackage{inputenc}
\usepackage[ruled,vlined]{algorithm2e}
\usepackage{verbatim}

\numberwithin{equation}{section}
\allowdisplaybreaks

\theoremstyle{plain}
\newtheorem{lemma}{Lemma}[section]

\newtheorem{theorem}{Theorem}[section]

\newtheorem{corollary}{Corollary}[theorem]
\theoremstyle{remark}
\newtheorem{remark}{\bf Remark}[section]
\theoremstyle{remark}

\DeclareMathOperator*{\argmin}{argmin}

\modulolinenumbers[5]

\journal{Journal of \LaTeX\ Templates}

\begin{document}

\begin{frontmatter}

\title{Extended Dynamic Mode Decomposition for Inhomogeneous Problems }
\author[mymainaddress]{Hannah Lu\fnref{email1}}
\fntext[email1]{email: \texttt{hannahlu{@}stanford.edu}}
\author[mymainaddress]{Daniel M. Tartakovsky\corref{mycorrespondingauthor}}
\cortext[mycorrespondingauthor]{Corresponding author}
\ead{tartakovsky@stanford.edu}
\address[mymainaddress]{Department of Energy Resources Engineering, Stanford University, Stanford, CA 94305, USA}

\begin{abstract}
Dynamic mode decomposition (DMD) is a powerful data-driven technique for construction of reduced-order models of complex dynamical systems. Multiple numerical tests have demonstrated the accuracy and efficiency of DMD, but mostly for systems described by partial differential equations (PDEs) with homogeneous boundary conditions. We propose an extended dynamic mode decomposition (xDMD) approach to cope with the potential unknown sources/sinks in PDEs. Motivated by similar ideas in deep neural networks, we equipped our xDMD with two new features. First, it has a bias term, which accounts for inhomogeneity of PDEs  and/or boundary conditions. Second, instead of learning a flow map, xDMD learns the residual increment by subtracting the identity operator. Our theoretical error analysis demonstrates the improved accuracy of xDMD relative to standard DMD. Several numerical examples are presented to illustrate this result.
\end{abstract}


\end{frontmatter}

\linenumbers

\section{Introduction}

Complexity of many, if not most, physical and biological phenomena and paucity of measurements undermine the reliability of purely statistical descriptors. Instead, models of such systems are inferred or ``learned'' from both observational and simulated data and reflect the fundamental laws of nature (e.g., conservation of mass and energy). Various sparse regression techniques~\cite{schmidt2009distilling, brunton2016discovering, schaeffer2017learning} use a proposed dictionary to ``discover'' the governing equations from observations. The dictionary, comprising plausible spatial and/or temporal derivatives of a state variable, provides functional approximations of different physical laws; dynamic mode decomposition (DMD) was used to inform the dictionary composition~\cite{li2017extended, korda2018convergence, williams2015data}. The data for sparse regression are allowed to be noisy~\cite{schaeffer2017sparse}, corrupted~\cite{tran2017exact}, and limited~\cite{schaeffer2018extracting}. Various flavors of deep neural networks (DNN) provide a related dictionary-based approach to PDE learning~\cite{raissi2019physics, tartakovsky2018learning, geneva2020modeling}. These and other techniques are as good as a dictionary on which they are based.

A conceptually different, dictionary-free, framework for data-informed predictions is to construct a surrogate (aka reduced-order) model, instead of learning a governing PDE. This framework is often classified as ``unsupervised learning'' and ``equation-free''. Much of the research in this field deals with dynamical systems, for which training data are generated by either ordinary differential equations (ODEs) or partial differential equations (PDEs) after spatial discretization. In this context, DMD can be used to construct an optimal linear approximation model for the unknown system~\cite{schmid2010dynamic} and to learn the unknown dynamics of chosen observables, rather than of the state itself~\cite{tu2014dynamic}. The latter is accomplished by utilizing the Koopman theory~\cite{koopman1931hamiltonian} in order to construct linear models on the observable space, instead of seeking for nonlinear models on the state space~\cite{brunton2017chaos}. Physics-guided selection of observables provides not only better accuracy~\cite{nathan2018applied, lu2019predictive, lu2020lagrangian}, but also a bridge between the understanding of data and physics. Likewise, DNN can be used to build nonlinear surrogate models for ODEs~\cite{
qin2019data, rudy2019deep} and PDEs~\cite{raissi2019physics, long2019pde, 
wu2020data}. DNN-based surrogates and reduced-order models (ROMs)~\cite{hesthaven2018non, pawar2019deep
} are invaluable in applications that require a large number of model solves, such as 
inverse modeling~\cite{raissi2019physics, mo2019deep, liu2019deep
} and uncertainty quantification~\cite{chan2018machine, karumuri2020simulator, tripathy2018deep, zhu2018bayesian}.

Our study contributes to this second framework by extending the range of applicability of DMD-based ROMs to dynamical systems described by inhomogeneous PDEs with inhomogeneous boundary conditions. Our extended dynamic mode decomposition (xDMD) borrows ideas from the recent work on residual neural networks (ResNet) to provide an optimal linear approximation model for such systems. Our generalization of the standard DMD includes two ingredients: an added bias term and residual learning. The first builds upon the generalized ResNet~\cite{chen2020gresnet} that introduces a bias term to model the dynamics described by underlying inhomogeneous ODEs. We extend this idea to systems described by inhomogeneous PDEs and prove the accuracy improvement induced by the added bias term. The second ingredient of xDMD is the learning of effective increments (i.e., the residual of subtracting identity from a flow map) rather than the flow map itself. Although mathematically equivalent to flow map learning, this strategy proved to be highly advantageous in practice and gained traction in the deep-learning community~\cite{ma2018}, including in its applications to equation recovery~\cite{qin2019data}. To the best of our knowledge, xDMD is first to fuse these two features and to provide a theoretical estimate. 

In section~\ref{sec2}, we provide a problem setup and provide a detailed formulation of xDMD. A formal proof of the accuracy improvements induced by the added bias term is presented in section~\ref{sec3}. A number of numerical experiments are presented in section~\ref{sec4} to evaluate the learning performance of xDMD in terms of representation, extrapolation, interpolation and generalizability. Key results, their implication for applications, and challenges and future work are summarized in section~\ref{sec5}.

\section{Problem Formulation and Extended DMD}
\label{sec2}

We consider a real-valued state variable $u(\mathbf x,t)$, whose dynamics is described by a boundary-value problem
\begin{equation}\label{2-1}
\left\{
\begin{aligned}
&\frac{\partial u}{\partial t} = \mathcal L(u)+ S(\mathbf x), \quad && (\mathbf x, t) \in \mathcal D \times \mathbb R^+;\\
&\mathcal B(u) = b(\mathbf x), && (\mathbf x, t) \in \partial \mathcal D \times \mathbb R^+;\\
&u(\mathbf x,0) = u_0(\mathbf x), && \mathbf x\in \mathcal D.
\end{aligned}
\right.
\end{equation}
Here $t$ denotes time; $\mathbf  x$ is the spatial coordinate; $\mathcal D \subset \mathbb R^d$ is the simulation domain bounded by the surface $\partial \mathcal D$; $\mathcal L$ is a (linear or nonlinear) differential operator that involves spatial derivatives; $\mathcal B$ is the boundary differential operator describing Dirichlet, Neumann, and/or Robin boundary conditions; $S(\mathbf x)$ and $b(\mathbf x)$ represent sources/sinks  and boundary functions, respectively; and $u_0(\mathbf x)$ is the initial state.

The simulation domain is discretized with a mesh consisting of $N$ elements. A suitable numerical approximation of~\eqref{2-1} yields a system of (coupled, nonlinear) ODEs,
\begin{equation}\label{2-2}
\left\{
\begin{aligned}
&\frac{\text d\mathbf u}{\text d t} =\mathbf f(\mathbf u, \mathbf s),&& \mathbf u, \mathbf s\in \mathbb R^N,\\
&\mathbf u (0) = \mathbf u_0,&&\mathbf u_0\in \mathbb R^N,
\end{aligned}
\right.
\end{equation}
where $\mathbf s$ comes from both $S(\mathbf x)$ and $b(\mathbf x)$. Let $\boldsymbol\Phi_{\Delta t}: \mathbb R^N\to \mathbb R^N$ denote a flow map, which relates the discretized system state $\mathbf u$ at time $t=0$ to that at time $t = \Delta t$, where $\Delta t$ is a (sufficiently small) time increment. Since $\mathbf s$ is independent of $t$ and acts as a set of parameters, the system~\eqref{2-2} is time-invariant.  Consequently, there exists a flow map $\boldsymbol\Phi$, depending only on the time difference $t-t_0$, which represents the solution to~\eqref{2-2} as
$\mathbf u(t;\mathbf u_0,t_0,\mathbf s) =\boldsymbol\Phi_{t-t_0}(\mathbf u_0;\mathbf s)$.

Our goal is to learn the  dynamic system $\mathbf f$, or, more precisely, its reduced-order surrogate, using $M$ temporal snapshots of the solutions. Let $\mathbf x^k \equiv \mathbf u(t_k)$ and $\mathbf y^k \equiv \mathbf u(t_k + \Delta t)$ with $k=1,\ldots,M$, where the time lag between the input and output states, $\Delta  t$, is assumed to be independent of $k$ for the sake of convenience. The simulation data consist of $M$ pairs $\{(\mathbf x^k,\mathbf y^k)\}_{k=1}^M$, such that
\begin{equation}\label{eq:y}
\mathbf {y}^k = \boldsymbol\Phi_{\Delta t}(\mathbf {x}^k;\mathbf s), \qquad k=1,\ldots,M.
\end{equation}

\begin{lemma}\label{lemma2-1}
Assume $\mathbf f$ to be Lipschitz continuous with a Lipschitz constant $L$ on a solution manifold $\mathcal M\subset \mathbb R^N$.  Define
\begin{equation}
\mathcal M_{\Delta t} = \{ \mathbf x\in \mathcal M: \boldsymbol\Phi_{\Delta t}(\mathbf x;\mathbf s) \in \mathcal M\}.
\end{equation}
Then, $\boldsymbol\Phi_{\Delta t}$ is Lipschitz continuous on $\mathcal M_{\Delta t}$. Specifically, for any $\mathbf z, \mathbf {\tilde z}\in \mathcal M_{\Delta t}$,
\begin{equation}
\|\boldsymbol\Phi_{\Delta t}(\mathbf z;\mathbf s)-\boldsymbol\Phi_{\Delta t}(\bold{\tilde z};\mathbf s)\|\leq \mathrm{e}^{L\tau}\|\mathbf z-\mathbf {\tilde z}\|, \qquad 0\leq \tau\leq \Delta t.
\end{equation}
\end{lemma}
\begin{proof}
The proof follows directly from the classical numerical analysis results in, e.g.,~\cite[p.~109]{stuart1998dynamical}.
\end{proof}

Lemma~\ref{lemma2-1} imposes requirements on the snapshots data pairs $\{(\mathbf x^k,\mathbf y^k)\}_{k=1}^M$: the number of data pairs $M$ should be sufficiently large, and  the data should be sufficiently rich for the data  space to cover the solution space of interest. These requirements are consistent with the core of the Koopman operator theory, which underpins the DMD algorithm, e.g.,~\cite[p.~47]{kutz2016dynamic} and others~\cite{williams2015data, tu2014dynamic, rowley2009spectral}. The error analyses of the DMD algorithms~\cite{korda2018convergence, lu2019predictive} also verifies the impact of the selection of observables on the success of Koopman methods.

\subsection{Standard DMD}
\label{sec:sDMD}

Given a  dataset of snapshots, $\{(\mathbf x^k,\mathbf y^k)\}_{k=1}^M$,  DMD constructs a best-fit linear operator $\mathbf A\in \mathbb R^{N\times N}$ such that
\begin{equation}\label{2-6}
\mathbf y^{k}\approx \mathbf A\mathbf x^{k}, \qquad k = 1,\ldots,M.
\end{equation}
Therefore, the matrix $\mathbf A$ is determined in a least square sense
\begin{equation}\label{2-7}
\mathbf A = \argmin_{\hat {\mathbf A}\in \mathbb R^{N\times N}}\frac{1}{M}\sum_{k=1}^M\|\mathbf y^{k}-\hat {\mathbf A}\mathbf x^{k}\|^2.
\end{equation}
Typically, one rewrites the dataset $\{(\mathbf x^k,\mathbf y^k)\}_{k=1}^M$ in  a matrix form,
\begin{equation}\label{2-8}
\mathbf X = \begin{bmatrix}
\mid&\mid&&\mid\\
\mathbf x^1&\mathbf x^2&\cdots&\mathbf x^{M}\\
\mid&\mid&&\mid
\end{bmatrix}_{N\times M} 
\quad\text{and}\quad
\mathbf Y = \begin{bmatrix}
\mid&\mid&&\mid\\
\mathbf y^1&\mathbf y^2&\cdots&\mathbf y^{M}\\
\mid&\mid&&\mid
\end{bmatrix}_{N\times M}.
\end{equation}
Then, $\mathbf A$ is computed as
\begin{equation}\label{2-9}
\mathbf A = \mathbf Y\mathbf X^\dag \qquad\qquad\qquad \text{(standard DMD)},
\end{equation}
where $\dag$ denotes the Moore-Penrose inverse.

\begin{remark}
The Moore-Penrose inverse is computed via singular value decomposition (SVD), which requires certain truncation criteria to maintain computational stability. In all our numerical tests, we use the default truncation in the \texttt{pinv} command of Matlab.
\end{remark}
\begin{remark}
In a typical DMD algorithm, e.g.,~\cite[p.~7]{kutz2016dynamic}, a reduced-order model $\tilde {\mathbf A}$ is derived by projecting $\mathbf A$ onto the proper orthogonal decomposition (POD) modes. Since the major goal of our study is to obtain a linear approximation model of inhomogeneous PDEs, for which standard DMD algorithms fail, we omit the order-reduction procedure for simplicity.
\end{remark}

\subsection{Generalized DMD}
\label{sec:gDMD}

In order to cope with potential inhomogeneity of the underlying dynamics, the following modification is made in~\cite{chen2020gresnet}:
\begin{equation}\label{2-10}
\mathbf y^{k}\approx \mathbf A_\text{g}\mathbf x^{k}+\mathbf b, \qquad k = 1,\ldots,M.
\end{equation}
The matrix $\mathbf A_\text{g}$ and the vector $ \mathbf b \in \mathbb R^N$ are computed by solving the optimization problem
\begin{equation}\label{2-11}
(\mathbf A_\text{g}, \mathbf b) = \argmin_{\hat {\mathbf A}\in \mathbb R^{N\times N},\hat{\mathbf b}\in \mathbb R^N}\frac{1}{M}\sum_{k=1}^M\|\mathbf y^{k}-\hat {\mathbf A}\mathbf x^{k}-\hat{\mathbf b}\|^2.
\end{equation}
Let us introduce
\begin{equation}\label{2-12}
\tilde{\mathbf X} := \begin{bmatrix}
\mathbf X\\
\mathbf 1
\end{bmatrix}_{(N+1)\times M}
\end{equation}
where $\mathbf 1 := [1,1,\cdots,1]$ is a vector of size $1\times M$. Then $\mathbf A_\text{g}$ and $\mathbf b$ are obtained by
\begin{equation}\label{2-13}
[\mathbf A_\text{g}, \mathbf b] = \mathbf Y\tilde{\mathbf X}^\dag \qquad\qquad\qquad \text{(generalized DMD or gDMD)}.
\end{equation}

\subsection{Residual DMD}
\label{sec:rDMD}

The residual DMD or rDMD borrows a key idea behind ResNet. The latter explicitly introduces the identity operator in a neural network and forces the network to approximate the ``residual" of the input-output map. Although mathematically equivalent, this simple transformation proved to improve  network performance and became increasingly popular in the machine learning community.

Writing $\mathbf A = \mathbf I + \mathbf B$, where $\mathbf I$ is the $(N \times N)$ identity matrix and $\mathbf B$ is the remainder, recasts~\eqref{2-6} as
\begin{equation}\label{2-14}
\mathbf y^k \approx \mathbf x^k+\mathbf B\mathbf x^k.
\end{equation}
The matrix $\mathbf B$ is determined by
\begin{equation}\label{2-15}
\mathbf B =(\mathbf Y-\mathbf X)\mathbf X^\dag \qquad\qquad\qquad \text{(residual DMD or rDMD)}.
\end{equation}
It provides an approximation of the ``effective increment''~\cite[definition 3.1]{qin2019data}, $\boldsymbol\varphi_{\Delta t}$, that is defined as
\begin{equation}\label{2-16}
\boldsymbol\varphi_{\Delta t}(\mathbf u;\mathbf f, \mathbf s) = \Delta t \, \mathbf f(\boldsymbol\Phi_\tau(\mathbf u;\mathbf s))
\end{equation}
for some $0\leq \tau\leq\Delta t$ such that
\begin{equation}\label{2-17}
\mathbf u(t+\Delta t) = \mathbf u(t)+ \boldsymbol\varphi_{\Delta t}(\mathbf u;\mathbf f,\mathbf s).
\end{equation}

\subsection{Extended DMD}
\label{sec:xDMD}

Combining the modification used in the previous two subsections, we arrive at our extended DMD or xDMD,
\begin{equation}\label{2-18}
\mathbf y^k \approx \mathbf x^k +\mathbf B_\text{g} \mathbf x^k+\mathbf b,
\end{equation}
where $\mathbf B_\text{g}$ and $\mathbf b$ are computed as
\begin{equation}\label{2-19}
[\mathbf B_\text{g},\mathbf b] =(\mathbf Y-\mathbf X)\mathbf {\tilde X}^\dag \qquad\qquad\qquad \text{(extended DMD or xDMD)}.
\end{equation}

\section{Relative Performance of Different DMD Formulations}
\label{sec3}

\begin{theorem}\label{thm1}
In the least square sense, gDMD in section~\ref{sec:gDMD} fits the $M$ snapshots data $\mathbf X$ and $\mathbf Y$ better than the standard DMD from section~\ref{sec:sDMD} does,  i.e.,
\begin{equation}\label{3-1}
\frac{1}{M}\sum_{k=1}^M\|\mathbf y^{k}-\mathbf A_\mathrm{g} \mathbf x^k-\mathbf b\|^2\leq \frac{1}{M}\sum_{k=1}^M\|\mathbf y^{k}-\mathbf A\mathbf x^k\|^2.
\end{equation}
\end{theorem}

\begin{proof}
The optimization problem~\eqref{2-11} gives rise to
\begin{equation}\label{3-2}
\begin{aligned}
&\frac{1}{M}\sum_{k=1}^M\|\mathbf y^{k}-\mathbf A_\mathrm{g} \mathbf x^k-\mathbf b\|^2\\
=&\min_{\hat {\mathbf A}\in \mathbb R^{N\times N},\hat{\mathbf b}\in \mathbb R^N}\frac{1}{M}\sum_{k=1}^M\|\mathbf y^{k}-\hat {\mathbf A}\mathbf x^{k}-\hat{\mathbf b}\|^2\\
\leq& \min_{\hat {\mathbf A}\in \mathbb R^{N\times N},\hat{\mathbf b}\in \mathbb R^N}\frac{1}{M}\sum_{k=1}^M\left(\|\mathbf y^{k}-\hat {\mathbf A}\mathbf x^{k}\|^2+\|\hat{\mathbf b}\|^2\right)\\
=& \min_{\hat {\mathbf A}\in \mathbb R^{N\times N}}\left(\frac{1}{M}\sum_{k=1}^M\|\mathbf y^{k}-\hat {\mathbf A}\mathbf x^{k}\|^2\right)+\min_{\hat{\mathbf b}\in \mathbb R^N}\|\hat{\mathbf b}\|^2\\
=&\frac{1}{M}\sum_{k=1}^M\|\mathbf y^{k}-\mathbf A\mathbf x^k\|^2+\min_{\hat{\mathbf b}\in \mathbb R^N}\|\hat{\mathbf b}\|^2.
\end{aligned}
\end{equation}
The inequality is derived by triangle inequality and the last equality is achieved by~\eqref{2-7}. Since the equality is achieved with $\hat {\bold b} = 0$,  gDMD is equivalent to the standard DMD only when the bias term $\bold b = 0$.
\end{proof}

\begin{remark}
Theorem~\ref{thm1} implies that xDMD from section~\ref{sec:xDMD} fits the $M$ snapshots data $\mathbf X$ and $\mathbf Y$ better than rDMD from section~\ref{sec:rDMD} in the least square sense, i.e.,
\begin{equation}\label{3-4}
\frac{1}{M}\sum_{k=1}^M\|\mathbf y^{k}-\mathbf x^k-\mathbf B_\text{g} \mathbf x^k-\mathbf b\|^2\leq \frac{1}{M}\sum_{k=1}^M\|\mathbf y^{k}-\mathbf x^k-\mathbf B\mathbf x^k\|^2.
\end{equation}
\end{remark}

\begin{corollary}\label{cor1}
Let $\mu_M$  be an empirical measure defined on a given dataset $\{\mathbf x^1,\cdots, \mathbf x^M\}$ by 
\begin{equation}\label{3-6}
\mu_M = \frac{1}{M}\sum_{k=1}^M\delta_{\mathbf x^k},
\end{equation}
where $\delta_{\mathbf x^k}$ denotes the Dirac measure at $\mathbf x^k$. Then, for any $\mathbf x\in \mathcal M_{\Delta t}$,
\begin{equation}\label{3-5}
\|\boldsymbol\Phi_{\Delta t}(\mathbf x)-\mathbf A_\mathrm{g} \mathbf x-b\|^2 \le \|\boldsymbol\Phi_{\Delta t}(\mathbf x)-\mathbf A \mathbf x\|^2 \;\;\mathrm{a.s.},
\end{equation}
i.e., the inequality~\eqref{3-5} holds in the sense of distribution.
\end{corollary}
\begin{proof}
The integral of a test function $g$ with respect to $\mu_M$ is given by
\begin{equation}\label{3-7}
\int_{\mathcal M}g(\mathbf x) \text d\mu_M(\mathbf x) = \frac{1}{M}\sum_{k=1}^M g(\mathbf x^k).
\end{equation}
It follows from~\eqref{3-1} and the definition of $\mathbf y^k$ in~\eqref{eq:y} that 
\begin{equation}\label{3-8}
\frac{1}{M}\sum_{k=1}^M\left(\|\boldsymbol\Phi_{\Delta t}(\mathbf x^k)-\mathbf A\mathbf x^k\|^2-\|\boldsymbol\Phi_{\Delta t}(\mathbf x^k)-\mathbf A_\text{g} \mathbf x^k-\mathbf b\|^2\right)  \ge 0.
\end{equation}
Thus, by virtue of~\eqref{3-7},
\begin{equation}\label{3-9}
\int_{\mathcal M} \left( \|\boldsymbol\Phi_{\Delta t}(\mathbf x)-\mathbf A \mathbf x\|^2 -\|\boldsymbol\Phi_{\Delta t}(\mathbf x)-\mathbf A_\text{g} \mathbf x - \mathbf b\|^2 \right) \text d\mu_M(\mathbf x) \ge 0.
\end{equation}
Hence, the inequality~\eqref{3-5} holds in the sense of distributions.
\end{proof}

\begin{remark}
By the same token, 
\begin{equation}\label{3-11}
\|\boldsymbol\Phi_{\Delta t}(\mathbf x)-\mathbf B_\text{g} \mathbf x - \mathbf b\|^2 \le \|\boldsymbol\Phi_{\Delta t}(\mathbf x)-\mathbf B \mathbf x\|^2, \mbox{ a.s.}
\end{equation}
\end{remark}

\begin{theorem}\label{thm2}
Suppose that the assumptions of Lemma~\ref{lemma2-1} hold, and further assume that
\begin{enumerate}
\item $\|\boldsymbol\Phi_{\Delta t}-\mathbf A\mathbf x\|_{L^\infty(\mathcal M_{\Delta t})}<+\infty$ and $\|\boldsymbol\Phi_{\Delta t}-\mathbf A_\mathrm{g} \mathbf x-\mathbf b\|_{L^\infty(\mathcal M_{\Delta t})}<+\infty$;
\item $\mathbf x^k,\mathbf y^k \in \mathcal M_{\Delta t}$ for $k = 1,\ldots,M$.
\end{enumerate}
Let $\mathbf u_\mathrm{DMD}^n$ and $\mathbf u_\mathrm{gDMD}^n$ denote solutions, at time $t^n \equiv t_0+n\Delta t$, of the DMD and gDMD models, respectively. Let $\mathbf u^n$ denote the true solution at time $t^n$, induced by the flow map $\boldsymbol\Phi_{\Delta t}$. Then errors of the DMD and gDMD models at time $t^n$,
\begin{equation}
\mathcal E_\mathrm{DMD}^n = \|\mathbf u^n- \mathbf {u}_\mathrm{DMD}^n\|^2 \qquad\text{and}\qquad 
\mathcal E^n_\mathrm{gDMD} = \|\mathbf u^n-\bold{u}_\mathrm{gDMD}^n\|^2,
\end{equation}
satisfy inequalities
\begin{equation}
\begin{aligned}
&\mathcal E_\mathrm{DMD}^n \leq (1 + \mathrm{e}^{L\Delta t })^n\mathcal E_\mathrm{DMD}^0 +\|\boldsymbol\Phi_{\Delta t} -\mathbf A\|_{L^\infty(\mathcal M)}\frac{(1+e^{L\Delta t})^n-1}{\mathrm{e}^{L\Delta t}},\\
&\mathcal E_\mathrm{gDMD}^n \leq (1 + \mathrm{e}^{L\Delta t })^n \mathcal E_\mathrm{gDMD}^0 +\|\boldsymbol\Phi_{\Delta t} -\mathbf A_\mathrm{g}-\mathbf b\|_{L^\infty(\mathcal M)}\frac{(1+\mathrm{e}^{L\Delta t})^n-1}{\mathrm{e}^{L\Delta t}}.
\end{aligned}
\end{equation}
\end{theorem}
\begin{proof}
The proof follows similar derivations as Theorem 4.3 in~\cite{qin2019data} using triangle inequality: 
\begin{equation}
\begin{aligned}
\mathcal E_\mathrm{DMD}^n & = \|\mathbf u^{n-1} +\boldsymbol\Phi_{\Delta t}(\mathbf u^{n-1})-\mathbf {u}_\mathrm{DMD}^{n-1} -\mathbf B\mathbf {u}_\mathrm{DMD}^{n-1}\|^2\\
&\leq \|\mathbf u^{n-1}-\bold{u}_\mathrm{DMD}^{n-1}\|+\|\boldsymbol\Phi_{\Delta t}(\mathbf u^{n-1}) -\mathbf B \mathbf {u}_\mathrm{DMD}^{n-1}\|^2\\
&\leq \|\mathbf u^{n-1}-\bold{u}_\mathrm{DMD}^{n-1}\|+\|\boldsymbol\Phi_{\Delta t}(\mathbf u_\mathrm{DMD}^{n-1}) -\mathbf B \mathbf {u}_\mathrm{DMD}^{n-1}\|^2+\|\boldsymbol\Phi_{\Delta t}(\mathbf {u}_\mathrm{DMD}^{n-1})-\boldsymbol\Phi_{\Delta t}(\mathbf u^{n-1})\|^2\\
&=\|\mathbf u^{n-1}-\bold{u}_\mathrm{DMD}^{n-1}\|+\|\boldsymbol\Phi_{\Delta t}(\mathbf u_\mathrm{DMD}^{n-1}) -\mathbf A \mathbf {u}_\mathrm{DMD}^{n-1}\|^2+\|\boldsymbol\Phi_{\Delta t}(\mathbf {u}_\mathrm{DMD}^{n-1})-\boldsymbol\Phi_{\Delta t}(\mathbf u^{n-1})\|^2\\
&\leq\|\mathbf u^{n-1}-\bold{u}_\mathrm{DMD}^{n-1}\|+\|\boldsymbol\Phi_{\Delta t}-\mathbf A\|_{L^\infty(\mathcal M_{\Delta t})}^2 +\mathrm{e}^{L\Delta t} \|\mathbf {u}_\mathrm{DMD}^{n-1}-\mathbf u^{n-1}\|^2\\
&=(1+\mathrm{e}^{L\Delta t})\mathcal E_\mathrm{DMD}^{n-1}+\|\boldsymbol\Phi_{\Delta t}-\mathbf A\|_{L^\infty(\mathcal M_{\Delta t})}^2 \\
&\leq (1+\mathrm{e}^{L\Delta t})\mathcal E_\mathrm{DMD}^{n-2}+\|\boldsymbol\Phi_{\Delta t}-\mathbf A\|_{L^\infty(\mathcal M_{\Delta t})}^2(1+(1+\mathrm{e}^{L\Delta t})) \\
&\leq\cdots\\
&\leq  (1+\mathrm{e}^{L\Delta t})^n\mathcal E_\mathrm{DMD}^{0}+\|\boldsymbol\Phi_{\Delta t}-\mathbf A\|_{L^\infty(\mathcal M_{\Delta t})}^2\sum_{k=0}^{n-1}(1+\mathrm{e}^{L\Delta t})^k \\
\end{aligned}
\end{equation}
A proof for the error bound for $\mathcal E_\mathrm{gDMD}^n$ is similar.
\end{proof}

\begin{remark}
The above error estimates indicate that gDMD has a tighter error bound than DMD a.s. because Corollary~\ref{cor1} indicates $\|\boldsymbol\Phi_{\Delta t} -\mathbf A_\text{g} -\mathbf b\|_{L^\infty(\mathcal M_{\Delta t})}^2\leq  \|\boldsymbol\Phi_{\Delta t} -\mathbf A\|_{L^\infty(\mathcal M_{\Delta t})}^2$ a.s..
\end{remark}

\begin{remark}
Similarly, xDMD has a tighter error bound than rDMD a.s.
\end{remark}

\section{Numerical Experiments}
\label{sec4}

We use a series of numerical experiments to demonstrate that xDMD outperforms other DMD variants and to validate our error estimates.\footnote{Additional numerical experiments are reported in the Supplemental Material.}  
Snapshots (training data) are obtained from reference solutions during time $[0,T]$, with input-output time-lag $\Delta t$, i.e.,
\begin{equation}
\mathbf X = \begin{bmatrix}
\mid&\mid&&\mid\\
\mathbf u^0&\mathbf u^1&\cdots&\mathbf u^{M}\\
\mid&\mid&&\mid
\end{bmatrix}, \qquad 
\mathbf Y = \begin{bmatrix}
\mid&\mid&&\mid\\
\mathbf u^1&\mathbf u^2&\cdots&\mathbf u^{M+1}\\
\mid&\mid&&\mid
\end{bmatrix}, 
\qquad  
T = (M+1)\Delta t.
\end{equation}
These datasets are assumed to be sufficiently large and rich to satisfy Lemma~\ref{lemma2-1}. 
We construct DMD and xDMD (and the other intermediate variants) by finding the best fit $\mathbf A$ or $(\mathbf B_\text{g}, \mathbf b)$, which yields a set of linear approximation models for the $\Delta t$ time-lag  input and output.  The ability of learning the unknown dynamics is tested in terms of 
\begin{itemize}
\item \emph{Representation}: Compare the difference between $\mathbf u^k$ and $\mathbf A^k\mathbf u^0$, or between $\mathbf u^k$ and $(\mathbf I+\mathbf B_\text{g})^k\mathbf u^0+\sum_{i = 0}^{k-1} (\mathbf I+\mathbf B_\text{g})^k\mathbf b$, for $k = 1,\ldots,M+1$. The error is essentially the least square fitting error, aka ``training error'' in machine learning.

\item \emph{Extrapolation}: Draw another set of reference solution $\{\mathbf u^k\}_{k=M+1}^{2(M+1)}$ from time interval $[T, 2T]$ following the same $\Delta t$ time-lag trajectory for the convenience of testing.  Compare the difference between $\mathbf u^k$ and $\mathbf A^k\mathbf u^0$, and between $\mathbf u^k$ and $(\mathbf I+\mathbf B_\text{g})^k\mathbf u^0+\sum_{i = 0}^{k-1} (\mathbf I+\mathbf B_\text{g})^k\mathbf b$, for $k = M+1,\ldots,2(M+1)$.

\item \emph{Interpolation}: Select a random subset of the dataset, i.e.,
\begin{equation}
\mathbf X_s = \begin{bmatrix}
\mid&\mid&&\mid\\
\mathbf u^{s_0}&\mathbf u^{s_1}&\cdots&\mathbf u^{s_m}\\
\mid&\mid&&\mid
\end{bmatrix}, \qquad 
\mathbf Y_s = \begin{bmatrix}
\mid&\mid&&\mid\\
\mathbf u^{s_0+1}&\mathbf u^{s_1+1}&\cdots&\mathbf u^{s_{m}+1}\\
\mid&\mid&&\mid
\end{bmatrix}, 
\end{equation}
where
$s_0 = 0$, $\{ s_1,\ldots, s_m\}\subset\{1,\cdots, M\}$, with $m<M$.
Then determine $\mathbf A$ and $(\mathbf B_\text{g}, \mathbf b)$ based on the selected dataset $\mathbf X_s$ and $\mathbf Y_s$.  Compare the difference between $\mathbf u^k$ and $\mathbf A^k\mathbf u^0$, and between $\mathbf u^k$ and $(\mathbf I+\mathbf B_\text{g})^k\mathbf u^0+\sum_{i = 0}^{k-1} (\mathbf I+\mathbf B_\text{g})^k\mathbf b$, for $k = 1,\ldots,M+1$. In our examples, the selected number of snapshots, $m$, is smaller than $M/2$.

\item \emph{Generalizability}: Determine $\mathbf A$ and $(\mathbf B_\text{g}, \mathbf b)$ from the datasets $\mathbf X$ and $\mathbf Y$, and obtain a linear approximation model of the discretized PDE. Compute another set of reference solutions $\{\mathbf v^k\}_{k=1}^{M+1}$ from a different initial input $\mathbf v^0\neq \mathbf u^0$ and the same boundary condition and source. Compare the difference between $\mathbf v^k$ and $\mathbf A^k\mathbf v^0$, and between $\mathbf v^k$ and $(\mathbf I+\mathbf B_m)^k\mathbf v^0+\sum_{i = 0}^{k-1} (\mathbf I+\mathbf B_m)^k\mathbf b$, for $k = 1,\ldots,M+1$. In our examples, the input $\mathbf v^0$ has completely different features than the training $\mathbf u^0$.

\item \emph{Accuracy}: The accuracy is compared in terms of the log relative errors,
\begin{equation}\label{eq:err}
\varepsilon_\text{DMD}^n := \lg\left(\frac{\|\mathbf u^n-\mathbf u_\text{DMD}^n\|^2}{\|\mathbf u^n\|^2}\right), \qquad
\varepsilon_\text{xDMD}^n := \lg\left(\frac{\|\mathbf u^n-\mathbf u_\text{xDMD}^n\|^2}{\|\mathbf u^n\|^2}\right).
\end{equation}

\end{itemize}
All comparisons between DMD and xDMD are made using the same dataset and the same SVD truncation criteria in the pseudo-inverse part (using the default truncation criteria in Matlab).

\subsection{Inhomogeneous PDEs}
\label{test1}
We start by examining the performance of the aforementioned DMD variants in learning a PDE with inhomogeneous source terms. Consider a one-dimensional diffusion equation with a source and homogeneous boundary conditions,
 \begin{equation}\label{eq:test1}
 \left\{
 \begin{aligned}
& \frac{\partial u}{\partial t} = 0.1 \frac{\partial^2 u}{\partial x^2} + S(x), \qquad x\in (0,1), \quad t>0.1; \\
&u(x,0) = \exp[-20(x-0.5)^2]; \\
&u_x(0,t) = 0, \quad u_x(1,t) = 0.
\end{aligned}\right.
\end{equation}
The reference solution is obtained by an implicit finite-difference scheme with $\Delta x = 0.01$ and $\Delta t = 0.01$.  Training datasets consist of $M = 80$ snapshots collected from $t = 0$ to $t = 0.8$. The extrapolation is tested from $t = 0.8$ to $t = 1.6$. The interpolation training set consists of $m = 20$ snapshots randomly selected from the $M = 80$ snapshots.

The left column of Figure~\ref{fig1} provides a comparison between the reference solution and its DMD and xDMD approximations in the three modes:  representation, extrapolation and interpolation. As predicted by the theory, DMD fails in all three modes. For a fixed time, the DMD error grows with  $x$, which is to be expected since standard DMD algorithms are not designed to handle inhomogeneous PDEs, such as~\eqref{eq:test1} in which the source term is $S(x) = x$. 
If a source term lies outside the span of the training data, as happens in this test, then it cannot be represented as a linear combination of the available snapshots. The DMD model always lies within the span of the training data, while the true solution grows out of that subspace because of the source. On the other hand, the xDMD model captures the true solution in all modes thanks to the bias term that accounts for the solution expansion outside the training data span. 

\begin{figure}[H]
\includegraphics{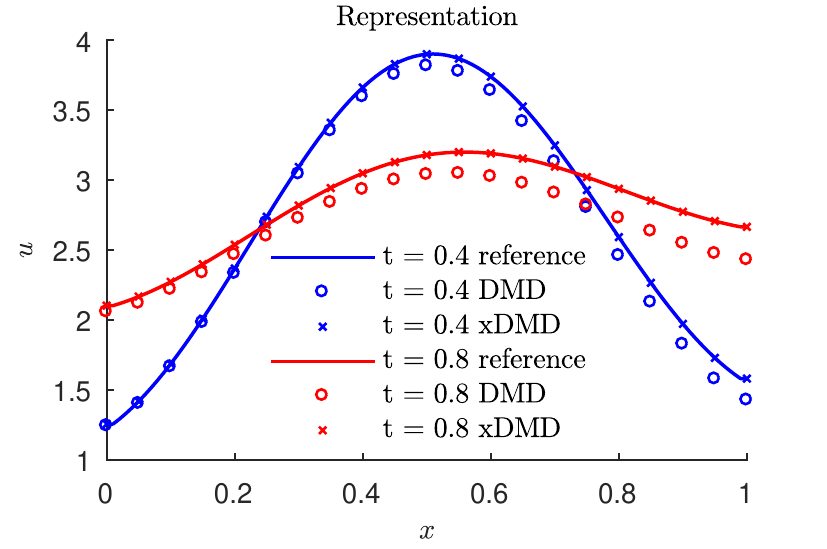}
\includegraphics{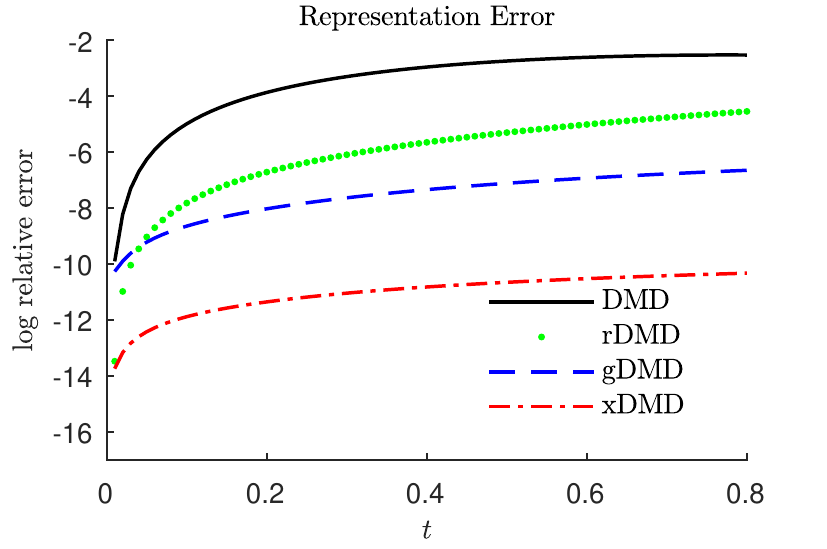}
\includegraphics{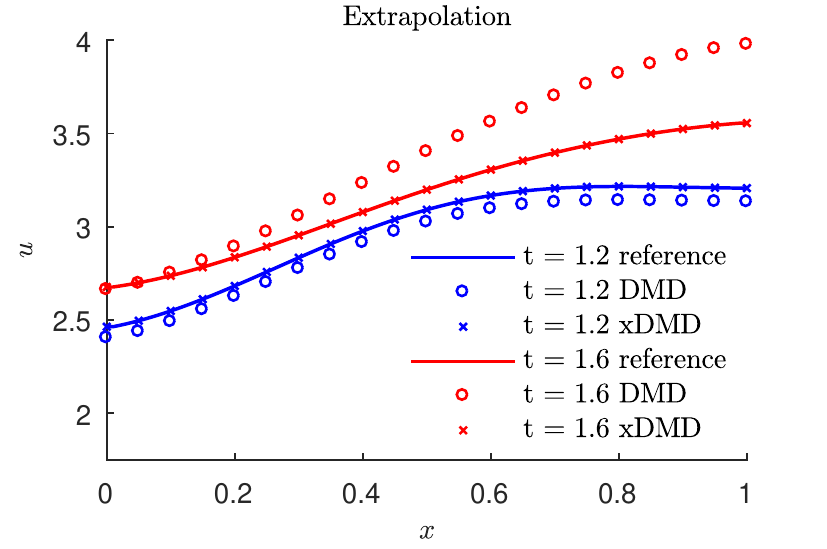}
\includegraphics{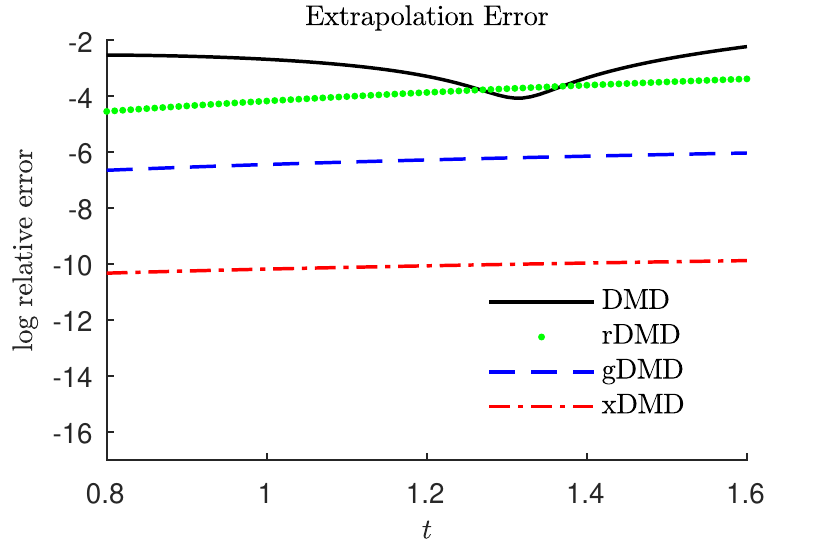}
\includegraphics{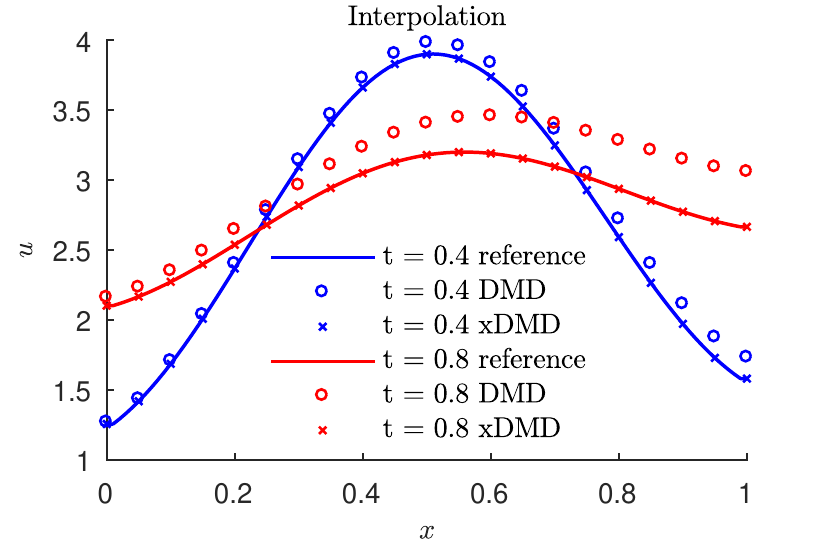}
\includegraphics{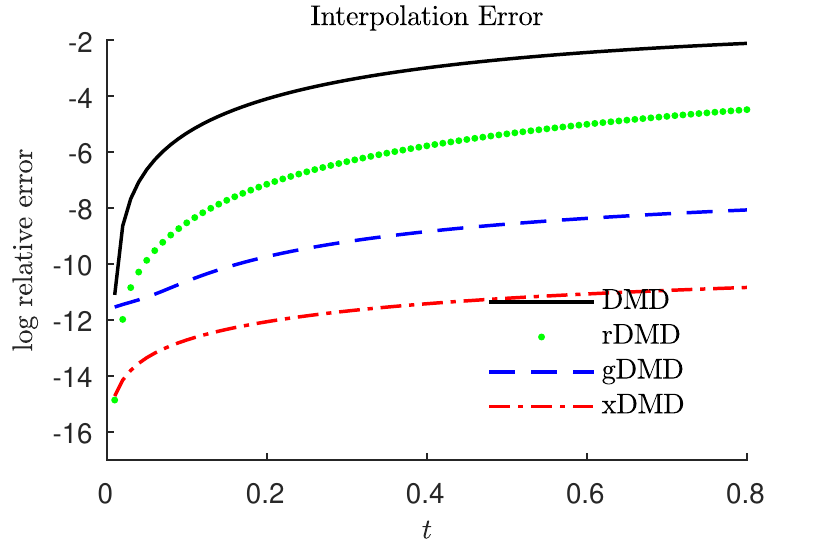}
\caption{Reference solution of~\eqref{eq:test1} and its DMD~\eqref{2-9} and xDMD~\eqref{2-19} approximations (left column), and the log relative error of the DMD~\eqref{2-9}, gDMD~\eqref{2-13}, rDMD~\eqref{2-15}, and xDMD~\eqref{2-19} models (right column).}
\label{fig1}
\end{figure}

The right column of Figure~\ref{fig1} shows the accuracy of the DMD~\eqref{2-9}, gDMD~\eqref{2-13}, rDMD~\eqref{2-15}, and xDMD~\eqref{2-19} models. Although DMD and rDMD are mathematically equivalent, the identity subtraction in rDMD reduces the solution error in all three modes (representation, extrapolation, and interpolation).  Addition of the bias term in xDMD contributes to further orders-of-magnitude reduction in the error, consistent with the theoretical proof in section~\ref{sec3}. In all modes, the proposed xDMD outperforms the other DMD variants by several orders of magnitude, achieving almost machine accuracy.

An added benefit of gDMD and xDMD is their ability to infer a source function, $S(x)$, in an inhomogeneous PDE from temporal snapshots of the solution (Figure~\ref{fig2}). Both methods recover $S(x)$, regardless of whether it is linear ($S = x$) or nonlinear ($S = \text{e}^x$), and have comparable errors. While DMD lamps together the differential operator and the source, gDMD and xDMD treat them separately. This endows them with the ability to learn both the operator (the system itself) and the source (external forces acting on the system), as long as the latter does not vary with time. This self-learning feature carries almost no extra computational cost.

\begin{figure}[H]
\includegraphics{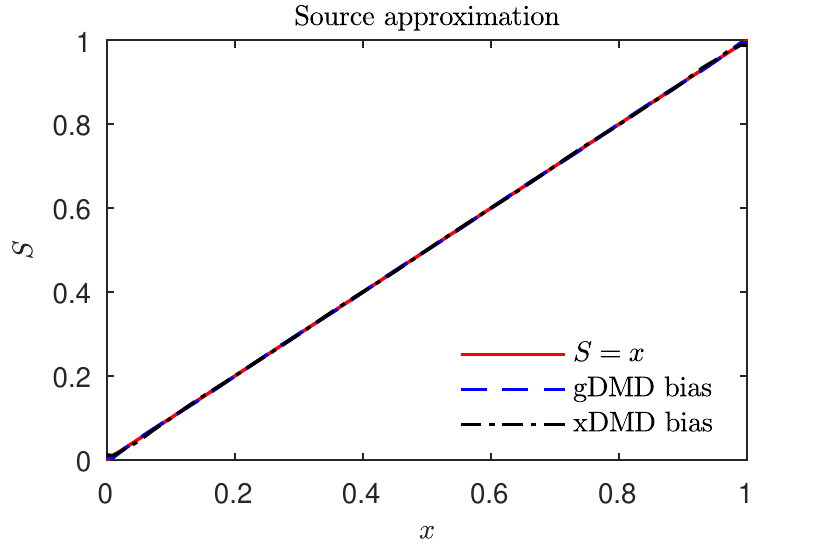}
\includegraphics{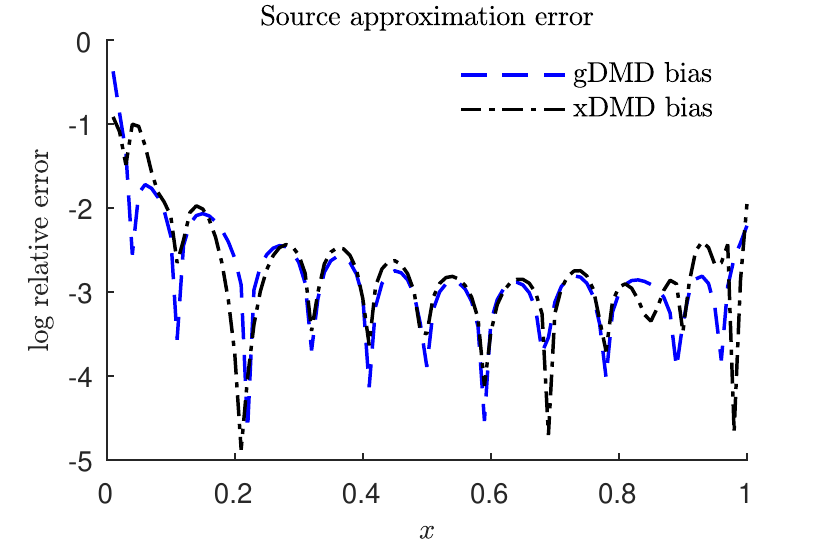}
\includegraphics{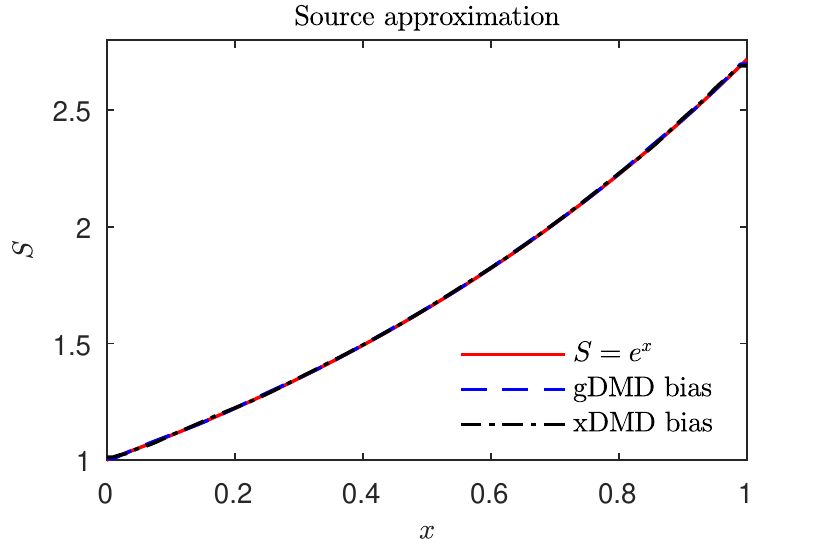}
\includegraphics{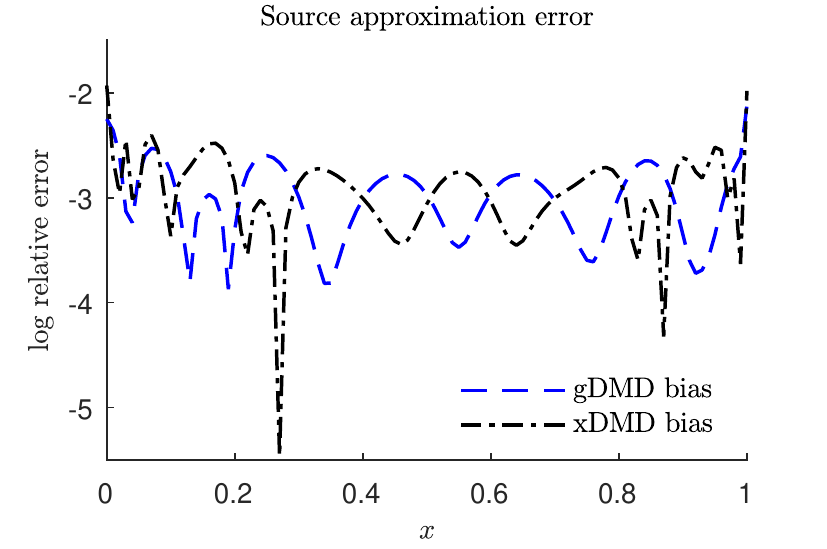}
\caption{Estimation of the source term $S(x) = x$ and $\text{e}^x$ in~\eqref{eq:test1} by gDMD and xDMD: the eyeball measure (left) and the log relative error (right).}
\label{fig2}
\end{figure}

\subsection{Inhomogeneous Boundary Conditions and Data Errors}
\label{test3}

Next, we examine the ability of DMD and xDMD to handle inhomogeneous boundary conditions and data errors. Consider a two-dimensional diffusion equation in a multi-connected domain $\mathcal D$ with inhomogeneous boundary conditions,
\begin{equation}\label{eq:test2}
\left\{
\begin{aligned}
&\frac{\partial u}{\partial t} =  \nabla^2 u,  \qquad (x,y) \in \mathcal D, \quad t\in (0,10000];\\
 &u(x,y,0) = 0;\\
 & u(0,y,t) =  3, \quad u(800,y,t) = 1, \quad \frac{\partial u}{\partial y}(x,0,t) = \frac{\partial u}{\partial y}(x,800,t) = 0, \quad 
  u(x,y,t) =  2 \;\;\;  \mbox{on $\partial \mathcal S$ (red)}.\\
 \end{aligned}\right.
 \end{equation}
The domain $\mathcal D$ is the $800 \times 800$ square with an S-shaped cavity (Figure~\ref{fig4}). The Dirichlet boundary conditions are imposed on the left and right sides of the square and the cavity surface. The top and bottom of the square are impermeable. The reference solution is obtained via  Matlab PDE toolbox on the finite-element mesh with $1633$ elements shown in Figure~\ref{fig4}. The solution from early transient time ($t = 2000$) until steady state ($t = 10000$) is presented in Figure~\ref{fig5}.

\begin{figure}[H]
\begin{center}
\includegraphics[trim={0 0.1cm 6cm 0},clip]{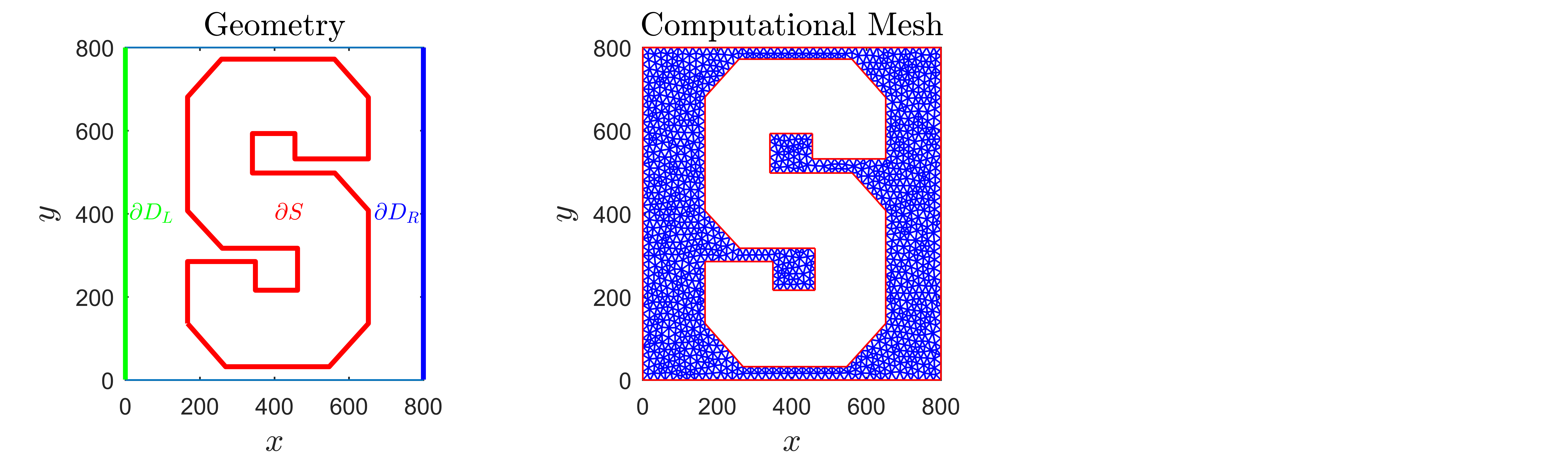}
\end{center}
\caption{Multi-connected simulation domain $\mathcal D$ (left) and the mesh used in the finite-element solution of~\eqref{eq:test2}.}
\label{fig4}
\end{figure}

\begin{figure}[H]
\includegraphics{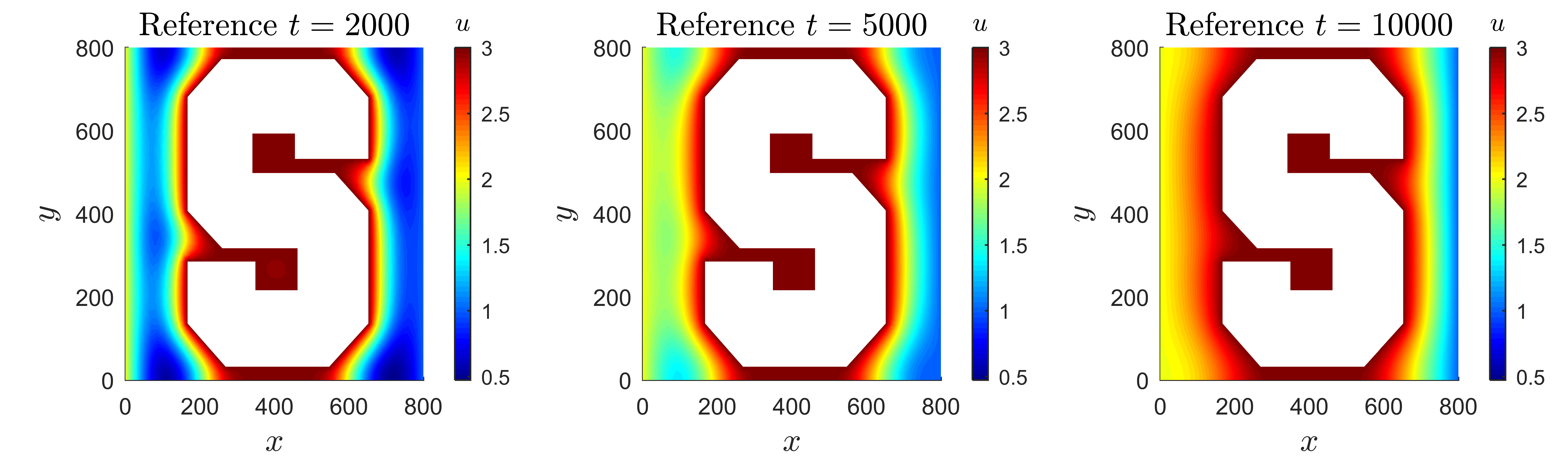}
\caption{The reference solution of~\eqref{eq:test2}, $u(x,y,t)$ at times $t = 2000$, $t = 5000$ and $t=10000$.}
\label{fig5}
\end{figure}

With the total simulation time (time sufficient to reach steady state) $t = 10000$, we generate snapshots spaced by $\Delta t = 5$ and use those to conduct four tests. First, the leading $M = 1200$ snapshots are used to inform DMD and xDMD and to ascertain their representation errors. Second, the DMD and xDMD models are deployed to extrapolate until $t = 10000$ and compare the extrapolation error of the two models. Third, randomly selected $M = 600$ snapshots from the first $1200$ snapshots are used for interpolation and to compare the interpolation error of DMD and xDMD. Finally, we repeat these representation/extrapolation/interpolation tests on data corrupted by addition of zero-mean white noise whose strengths at any $(x,t)$ is $0.1\%$ of the nominal value of $u(x,t)$ at that point.

\begin{figure}[H]
\begin{center}
\includegraphics{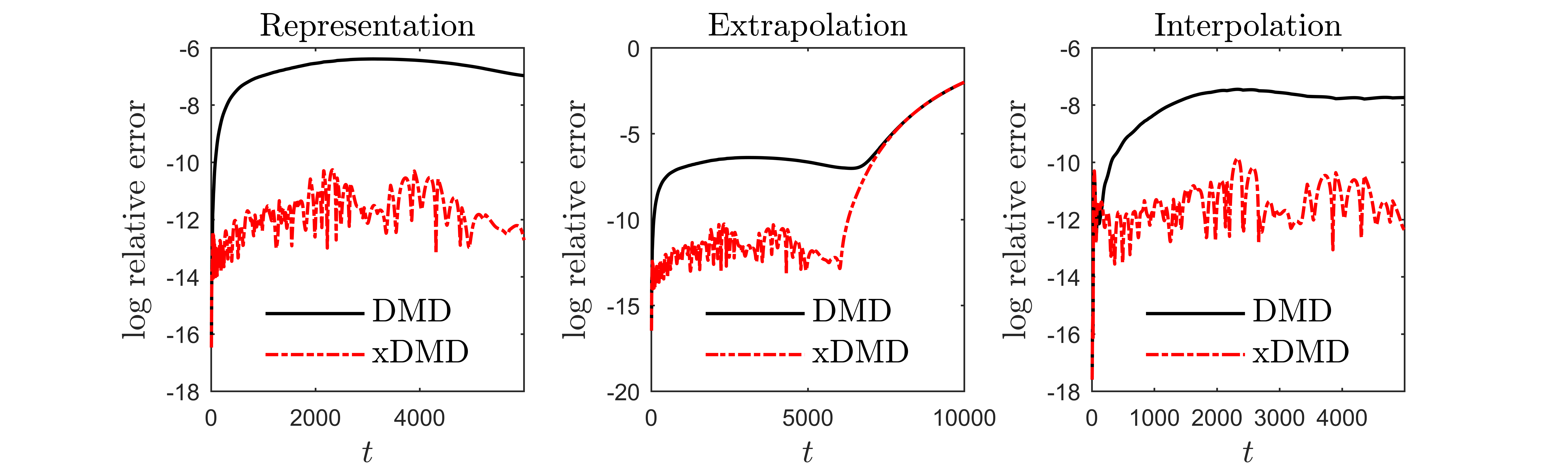}
\includegraphics{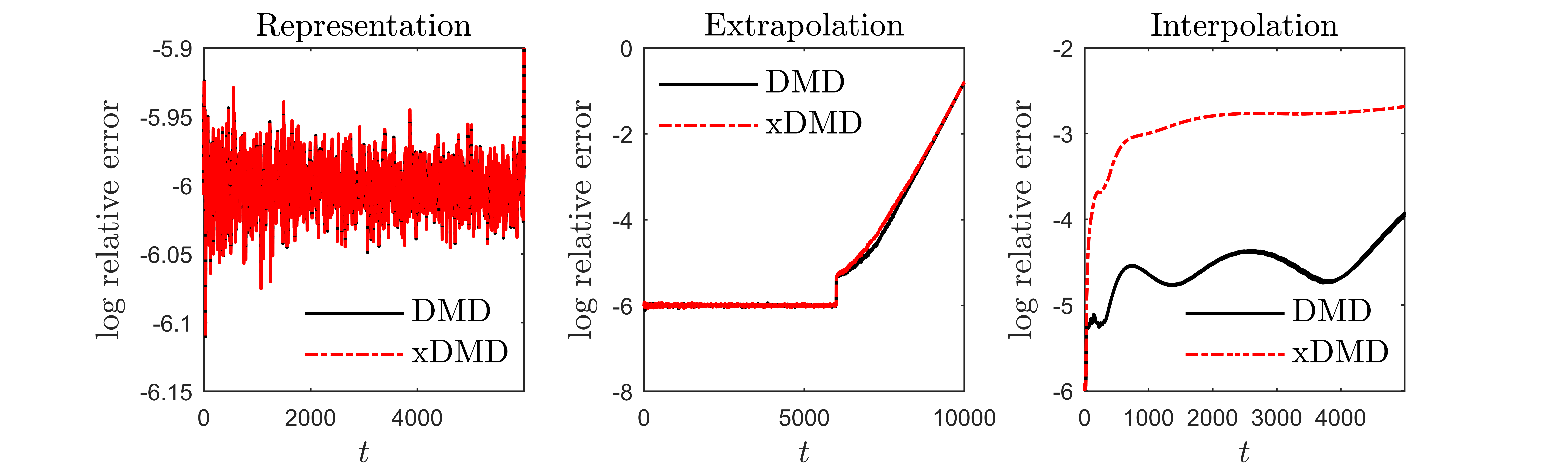}
\end{center}
\caption{Dependence of the log relative error of the DMD and xDMD models on time in the representation, extrapolation and interpolation modes. These errors are reported for noiseless data (top row) and data corrupted by addition of zero-mean white noise whose strengths at any $(x,t)$ is $0.1\%$ of the nominal value of $u(x,t)$ at that point (bottom row).}
\label{fig6}
\end{figure}

Figure~\ref{fig6} reveals that, for noiseless data, the accuracy of xDMD is orders of magnitude higher than that of DMD in the representation and interpolation modes; in the extrapolation mode, the error is dominated by the extrapolation error, which increases with time, but xDMD is still about $9\%$ more accurate than DMD at later times. However, in the presence of measurement noise, xDMD has no better performance than DMD, it is even less accurate in the extrapolation and interpolation regimes. This sensitivity to noise mirrors the over-fitting issue in machine learning: models with more parameters fit the limited number of available data (solution snapshots) too closely and, consequently, fail to fit additional data or to reliably predict future observations. Since xDMD has more parameters than DMD due to the bias term, one should expect the former to be more sensitive to noise than the latter.

\subsection{Coupled Nonlinear PDEs}

Common sense suggests that the success of linear models, such as DMD and xDMD, to approximate nonlinear dynamics is not guaranteed. In machine learning, data augmentation by feature map is widely used to deal with the nonlinearity. Similarly, judiciously chosen observables play a crucial role in the success of data-driven (DMD) modeling~\cite{korda2018convergence, lu2020lagrangian, mezic2013analysis}. The selection of observables requires prior knowledge of the underlying process, which is out of scope of this study. Instead, we assume no prior knowledge and apply no data augmentation, i.e., our observables are the state itself. To satisfy the assumptions in Lemma~\ref{lemma2-1}, we restrict our attention to nonlinear PDEs, whose solutions are confined in certain subspace $\mathcal M$. Our numerical experiments deal with the two-dimensional viscous Burgers equation (reported in the Supplemental Material) and the two-dimensional Navier-Stokes equations. The goal of these tests is to assess the ability of DMD and xDMD to learn complex flow maps.

We consider two-dimensional flow of an incompressible fluid with density $\rho = 1$ and dynamic viscosity $\nu = 1/600$ (these and other quantities are reported in consistent units) around an impermeable circle of diameter $D = 0.1$. The flow, which takes place inside a rectangular domain $\mathcal D = \{ \mathbf x = (x,y)^\top : (x,y) \in [0,2] \times [0,1] \}$, is driven by an externally imposed pressure gradient; the center of the circular inclusion is $\mathbf x_\text{circ} = (0.3,0.5)^\top$. Dynamics of the three state variables, flow velocity $\mathbf u(\mathbf x,t) = (u,v)^\top$ and fluid pressure $p(\mathbf x,t)$, is described by the two-dimensional Navier-Stokes equations,
\begin{equation}\label{eq:NS}
\left\{
\begin{aligned}
&\frac{\partial u}{\partial x}+\frac{\partial v}{\partial y} = 0; \\
&\frac{\partial u}{\partial t}+u\frac{\partial u}{\partial x}+v\frac{\partial u}{\partial y} = -\frac{1}{\rho}\frac{ \partial p}{ \partial x}+\nu\left(\frac{\partial^2 u}{\partial x^2}+\frac{\partial^2 u}{\partial y^2}\right), \qquad \mathbf x \in \mathcal D, \quad t > 0; \\
&\frac{\partial v}{\partial t}+u\frac{\partial v}{\partial x}+v\frac{\partial v}{\partial y} = -\frac{1}{\rho}\frac{\partial p}{\partial y}+\nu\left(\frac{\partial^2 v}{\partial x^2}+\frac{\partial^2 v}{\partial y^2}\right);
\end{aligned}
\right.
\end{equation}
subject to the initial conditions $\mathbf u(x,y,0) = (0,0)^\top$ and $p(x,y,0) = 0$; and the boundary conditions $\mathbf u(0,y,t) = (1,0)^\top$, 
$p(2,y,t) = 0$; and $\mathbf u(x,0,t) = \mathbf u(x,1,t) = (0,0)^\top$. This combination of parameters results in the Reynolds number $\text{Re} = 1200$.


The reference solution is obtained with the Matlab code~\cite{NScode}, which implements a finite-difference scheme on the staggered grid with $\Delta x = \Delta y = 0.02$ and $\Delta t = 0.0015$. Our observable (quantity of interest) is the magnitude of the flow velocity, $U(x,y,t) = \sqrt{u^2+v^2}$. Visual examination of the solution $U(x,y,t)$ reveals it to be periodic from $t = 7.5$ to $t = 15$ (the simulation horizon),  i.e., the solution is confined in a fixed subspace $\mathcal M$.  We collect $M = 2500$ snapshots of $U$ from $t = 7.5$ to $t = 11.25$ into a training dataset, from which DMD and xDMD learn the nonlinear dynamics. The discrepancy between the reference solution and its fitting with the DMD and xDMD models is the representation error. 

\begin{figure}[H]
\includegraphics{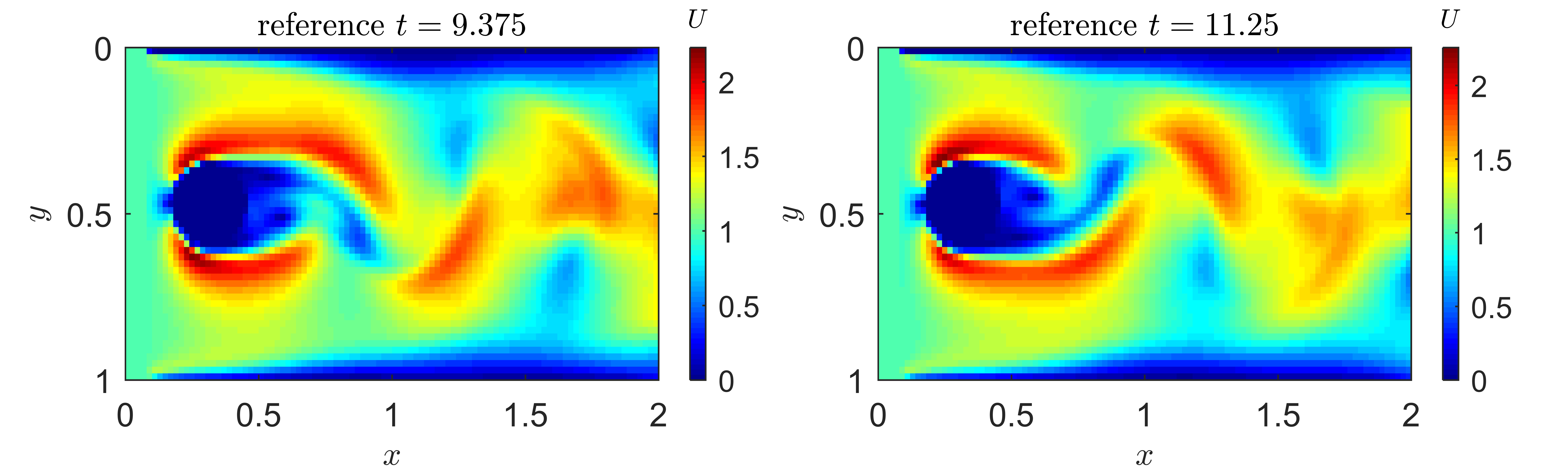}
\includegraphics{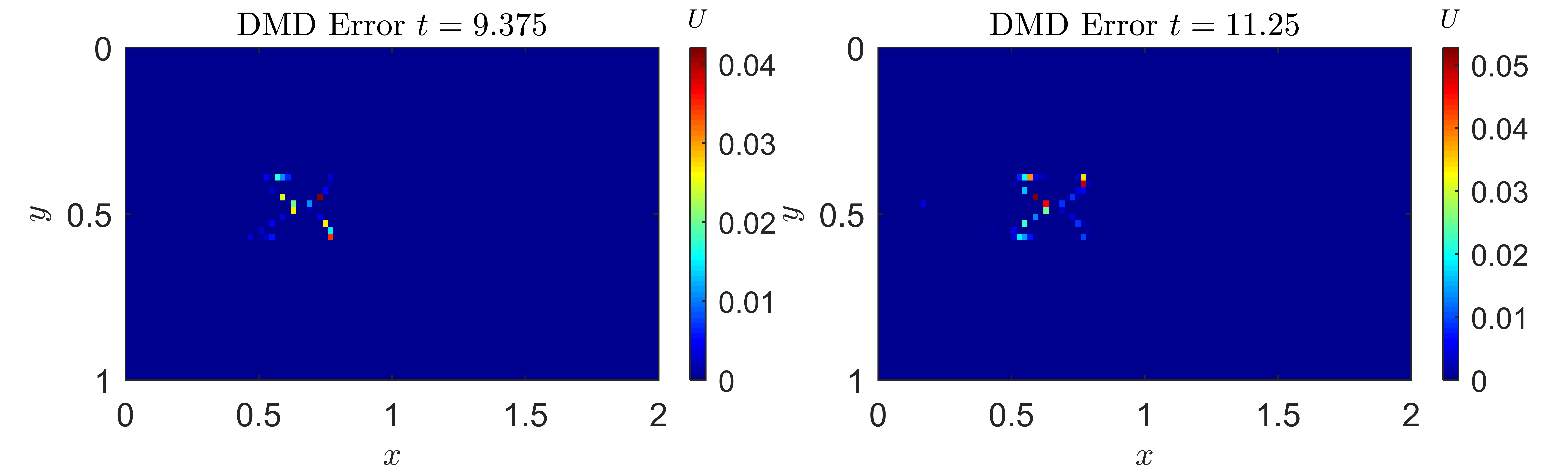}
\includegraphics{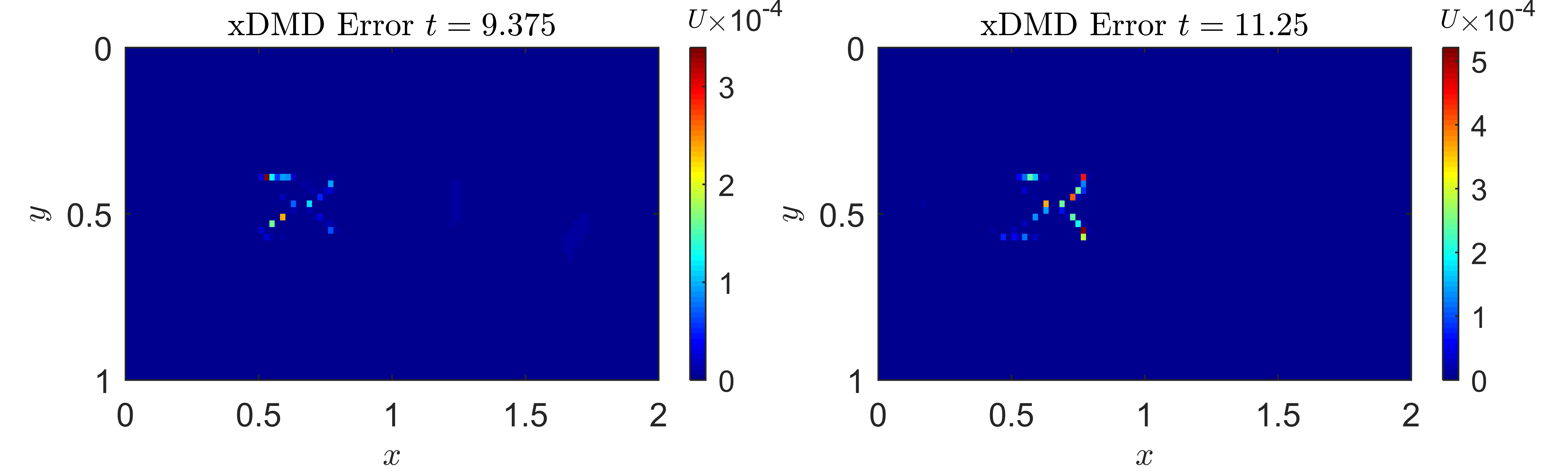}
\caption{Velocity magnitude $U = \sqrt{u^2+v^2}$ of incompressible flow with the Reynolds number $\text{Re} = 1200$ around an impermeable circle predicted by solving numerically the two-dimensional Navier-Stokes equations~\eqref{eq:NS} (top row) and by using the DMD and xDMD models in the representation mode. The representation errors~\eqref{eq:err} for these two approximations are displayed in the second and third rows, respectively. }
\label{fig10}
\end{figure}

The first row of Figure~\ref{fig10} depicts the spatial distribution of the flow speed $U$, at times $t = 9.38$ and $t = 11.25$, computed with the (reference) solution of the Navier-Stokes equations~\eqref{eq:NS}.  Both DMD and xDMD fit the nonlinear flow data using a linear approximation with satisfactory accuracy (the last two rows of Fig.~\ref{fig10}). The errors are confined to the circle's wake, with xDMD being two orders of magnitude more accurate than DMD.

\begin{figure}[H]
\includegraphics{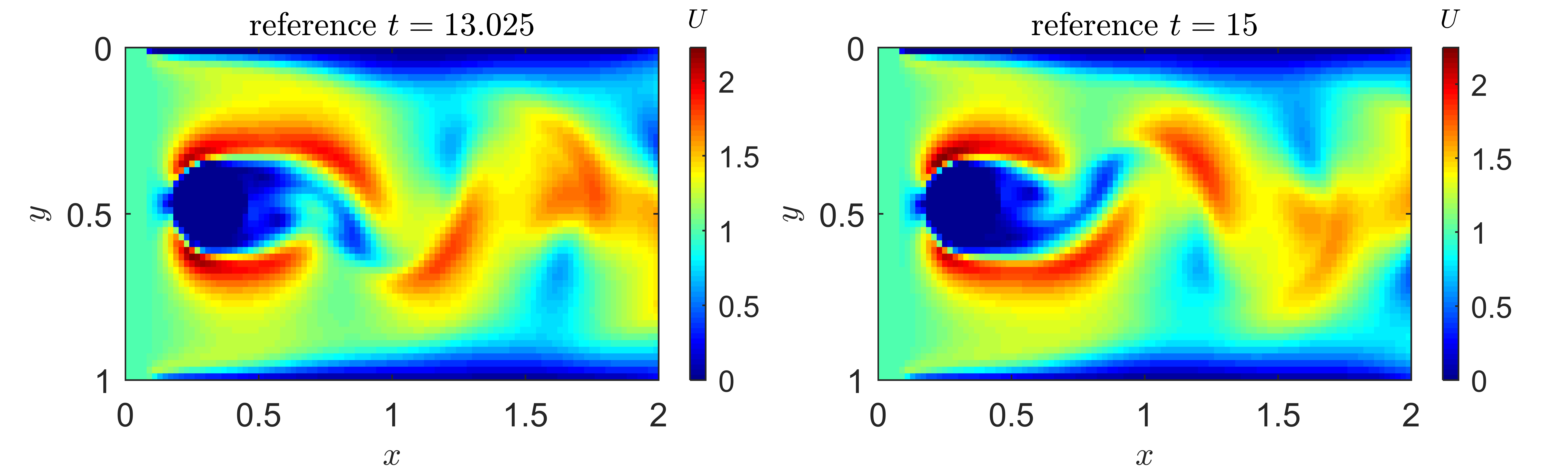}
\includegraphics{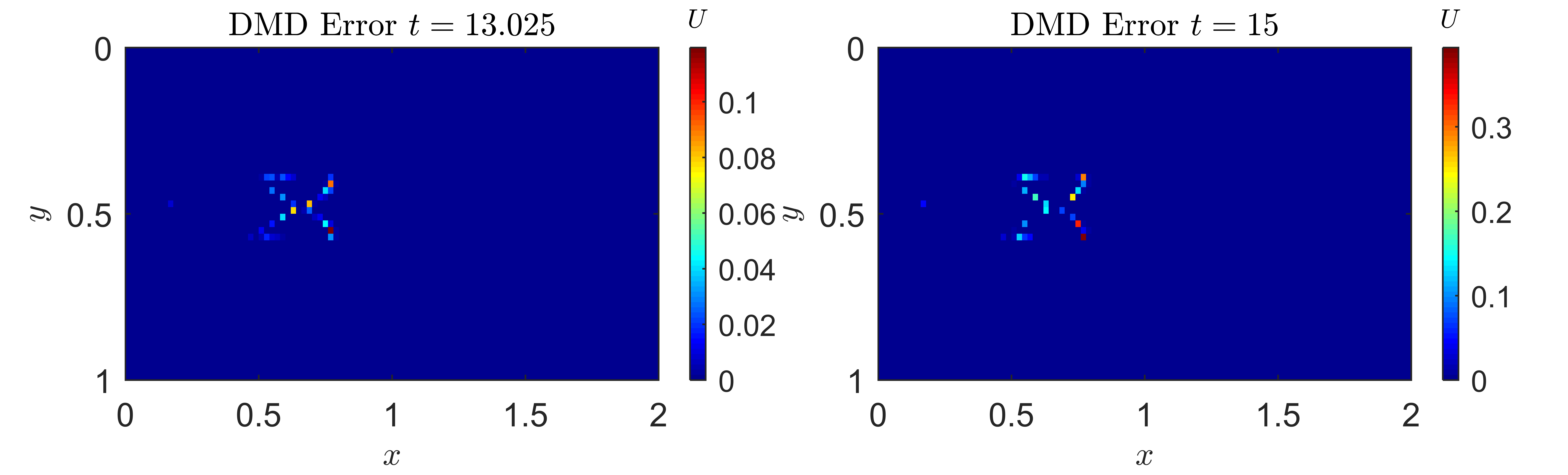}
\includegraphics{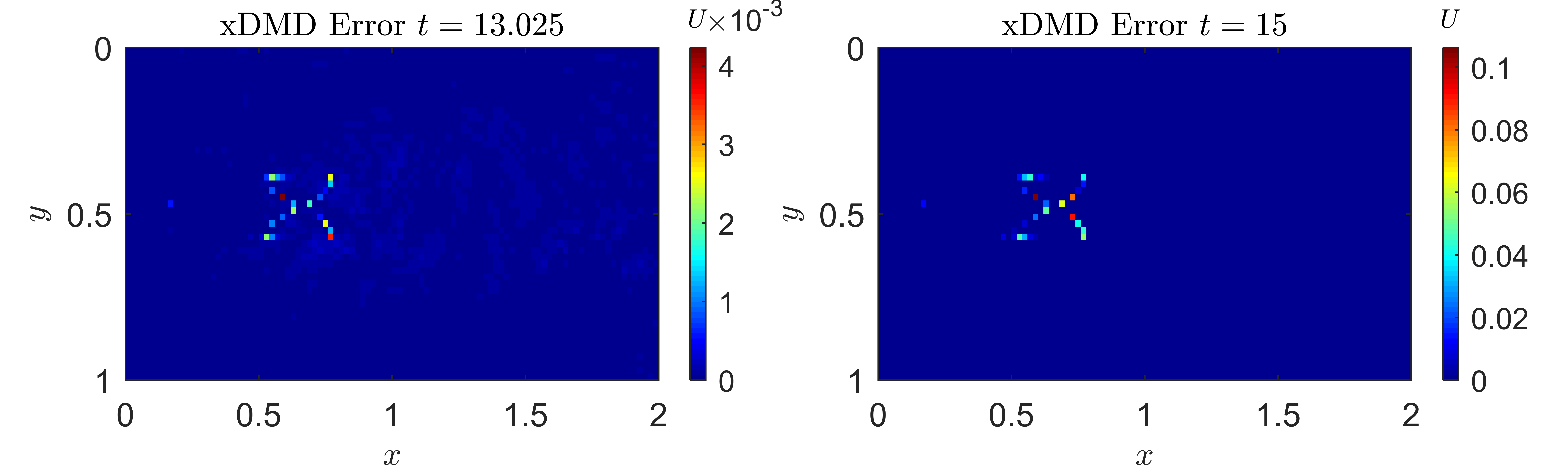}
\caption{Velocity magnitude $U = \sqrt{u^2+v^2}$ of incompressible flow with the Reynolds number $\text{Re} = 1200$ around an impermeable circle predicted by solving numerically the two-dimensional Navier-Stokes equations~\eqref{eq:NS} (top row) and by using the DMD and xDMD models in the extrapolation mode. The extrapolation errors~\eqref{eq:err} for these two approximations are displayed in the second and third rows, respectively.}
\label{fig11}
\end{figure}

Next, we use the learned DMD and xDMD models in the extrapolation mode, i.e., to predict $U(x,y,t)$ within the time interval from $t = 11.25$ to $t = 15$. As shown in Figure~\ref{fig11}, both DMD and xDMD yield accurate extrapolation, which should be expected due to the periodic behavior of the solution. Although the accuracy in extrapolation is diminished for both methods, xDMD remains more accurate than DMD at different extrapolation times.

\begin{figure}[H]
\includegraphics{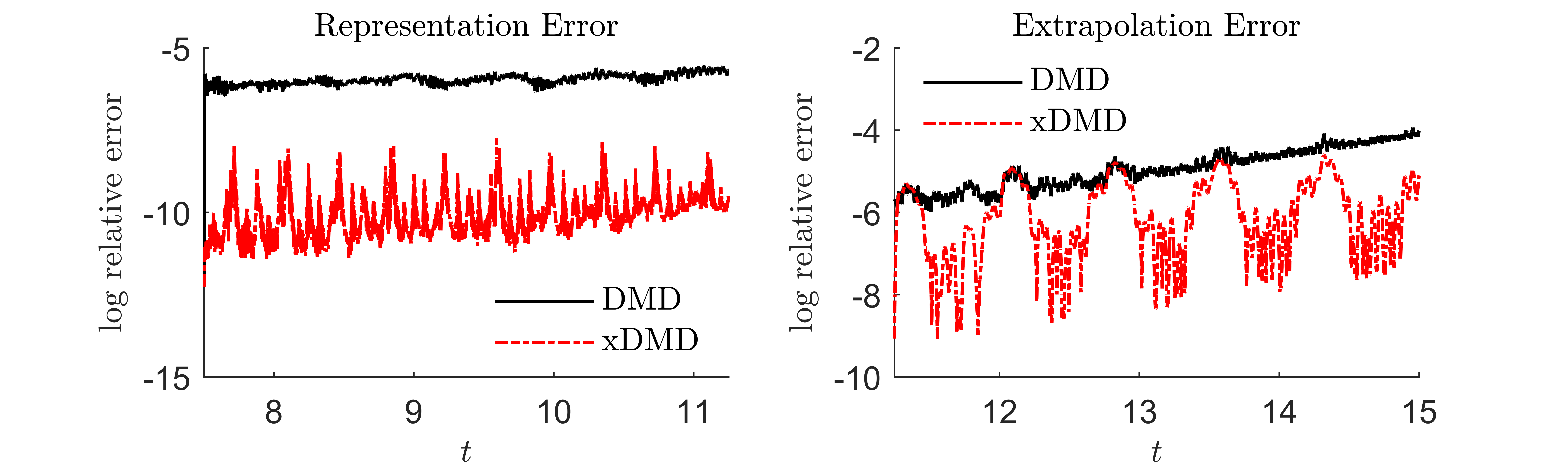} 
\caption{Dependence of log relative error of the DMD and xDMD models on time in the representation and extrapolation modes.}
\label{fig12}
\end{figure}

Finally, Figure~\ref{fig12} exhibits the log relative error of the two methods as function of time. In the representation mode, both DMD and xDMD have nearly steady small fitting error, fluctuating about $10^{-10}$ for xDMD and $10^{-6}$ for DMD. The observation of xDMD's higher accuracy in fitting the data is consistent with Theorem~\ref{thm1}. Similarly, the extrapolation error of DMD and xDMD validates Theorem~\ref{thm2}. Although both extrapolation errors grow slowly, the xDMD error exhibits a periodic pattern, indicating that the xDMD linear model is able to capture the detailed periodic feature of the true flow better. Once an accurate linear representation of the nonlinear flow is available, one can conduct spatiotemporal mode analysis, reduced-order modeling and accelerated simulations. Table~\ref{table1} collates computational times of simulating the reference solution and the linear approximation models. Further reduction in computation cost can be achieved by constructing reduced-order models using eigen-decomposition in DMD and xDMD.

\begin{table}[H]
\begin{center}
\begin{tabular}{|c|c|c|c| } 
 \hline
 &Simulation & DMD& xDMD\\ 
 \hline
Computational time (sec)&$29.0776$&$2.1352$& $2.1654$\\ 
\hline
Relative error & -- &$2.0515\times 10^{-5}$& $3.1193\times 10^{-6}$ \\
 \hline
\end{tabular}
\end{center}
\caption{Computational time and relative error for the reference solution and the DMD and xDMD models.}
\label{table1}
\end{table}

\subsection{Generalizability to New Inputs}\label{sec:generalizability}
Generalizability refers to a model's ability to adapt properly to new, previously unseen data, drawn from the same distribution as the one used to create the model. With validated generalizability, a DMD or xDMD model can be employed as a surrogate to accelerate, e.g., expensive Markov Chain Monte Carlo (MCMC) sampling used in inverse problems. 
%
%
A typical setting for this type of problems is solute transport in groundwater flow, whose steady-state Darcy velocity (flux) $\mathbf q(\mathbf x) = - K\nabla h $ is computed from the groundwater flow equation
 \begin{equation}
\nabla\cdot(K\nabla h) = 0.
\end{equation}
Here $h(\mathbf x)$ is the hydraulic head, and $K(\mathbf x)$ is the hydraulic conductivity of a heterogeneous subsurface environment; in our simulations we use a rectangular simulation domain $\mathcal D = \{ \mathbf x = (x,y)^\top : (x,y)\in [0,128]\times [0,64] \}$ and the $K(\mathbf x)$ field  in Figure~\ref{fig18} (these and other quantities are expressed in consistent units). The boundary conditions are $h(x=0,y) = 1, h(x = 128,y) = 0$ and impermeable on $y = 0, y = 64$.

\begin{figure}[H]
\begin{center}
\includegraphics{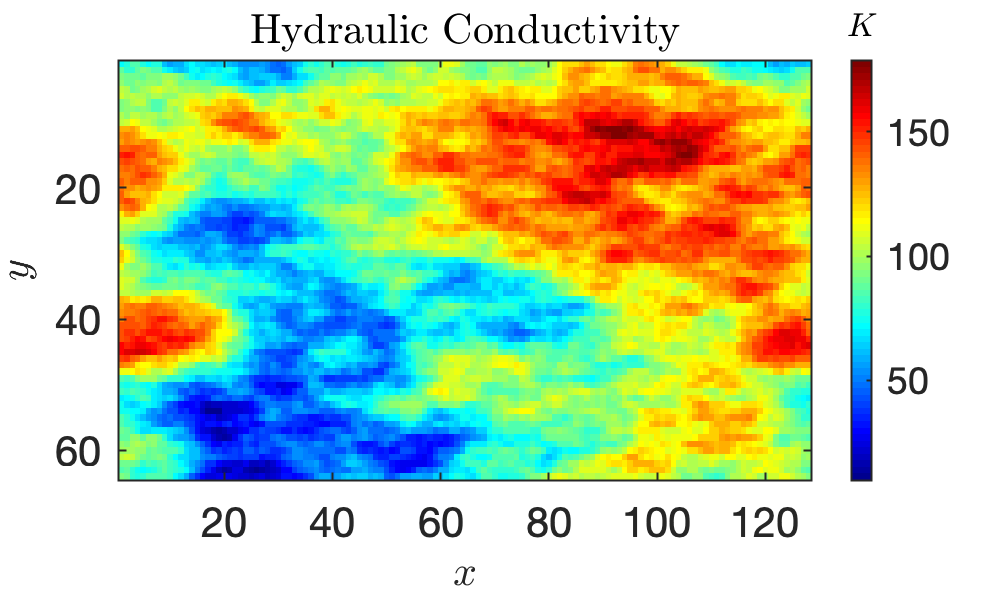}
\end{center}
\caption{Spatial distribution of hydraulic conductivity $K(\mathbf x)$ used in our simulations.}
\label{fig18}
\end{figure}

The resulting macroscopic velocity $\mathbf v(\mathbf x) = \mathbf q / \omega$, with $\omega$ denoting the porosity, is then used in the advection-dispersion equation to calculate the contaminant concentration $u(x,y,t)$:
\begin{equation}\label{eq:ade}
\frac{\partial u}{\partial t} + \mathbf v\cdot\nabla u=  \nabla\cdot(\mathbf D\nabla u), \qquad \mathbf x \in \mathcal D, \quad t \in (0,T],  
\end{equation}
with $T = 80$. In general, the dispersion coefficient $\mathbf D$ is a second-rank semi-positive definite tensor, whose components depend on the magnitude of the flow velocity, $|\mathbf u|$. Here, for illustrative purposes, we treat it as the identity matrix, $\mathbf D = \mathbf I$. The boundary conditions for~\eqref{eq:ade} are $u(0,y,t) = 0.2$ and $\partial_x u(128,y,t) = \partial_y u(x,0,t) = \partial_y u(x,64,t) = 0$. The training is done for the initial condition $u(x,y,0) = u_\text{in}(x,y)$ with
\begin{align}
u_\text{in} = s \exp[-(x-x_\text{s})^2-(y-y_\text{s})^2],
\end{align}
where $s = 100$ and the coordinates of the plume's center of mass, $(x_s,y_s)$ are treated as independent random variables with uniform distributions, $x_s\sim\mathcal U[0,25]$ and $y_s\sim \mathcal U[0,64]$. We generate $N_\text{MC}$ realizations of the pairs $(x_s,y_s)$ and evaluate the corresponding initial conditions $u_\text{in}^{(n)}(\bold x)$ for $n = 1,\dots, N_\text{MC}, \ N_\text{MC} = 2000$. For each of these realizations, \eqref{eq:ade} is solved\footnote{The reference solutions are obtained with the groundwater flow simulator MODFLOW and the solute transport simulator MT3DMS, both ran on a uniform mesh $\Delta x = \Delta y = 1$.}  to compute our quantity of interest, the concentration field $u_T^{(n)}(\bold x) \equiv u^{(n)}(\bold x, T)$. The matrix pairs $\{u_\text{in}^{(n)},u_T^{(n)}\}_{n=1}^{N_\text{MC}}$ are arranged into
\begin{equation}\label{eq:4-10}
\mathbf X = \begin{bmatrix}
\mid&\mid&&\mid\\
\mathbf x^1&\mathbf x^2&\cdots&\mathbf x^{N_\text{MC}}\\
\mid&\mid&&\mid
\end{bmatrix}
\quad\text{and}\quad
\mathbf Y = \begin{bmatrix}
\mid&\mid&&\mid\\
\mathbf y^1&\mathbf y^2&\cdots&\mathbf y^{N_\text{MC}}\\
\mid&\mid&&\mid
\end{bmatrix}
\end{equation}
where $\bold x^n$ is vectorized $u_\text{in}^{(n)}$ and $\bold y^n$ is vectorized $u_T^{(n)}$. Finally, the DMD and xDMD models are deployed to learn the flow map $\boldsymbol \Phi_{\Delta t}$ with the time lag $\Delta t = T$.

\begin{figure}[H]
\includegraphics{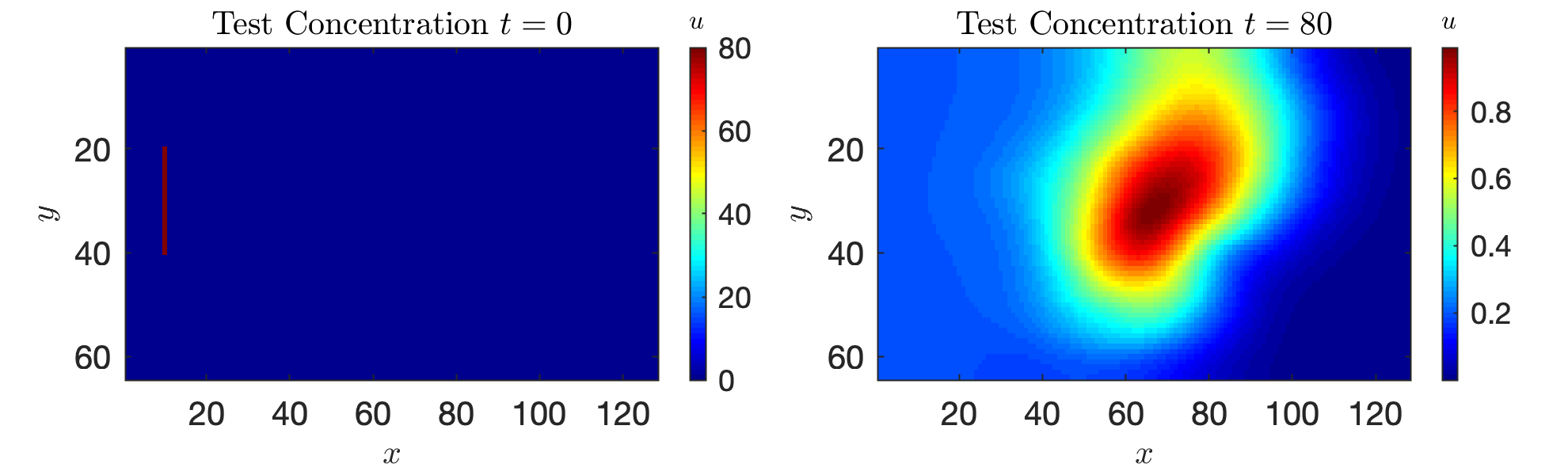}
\includegraphics{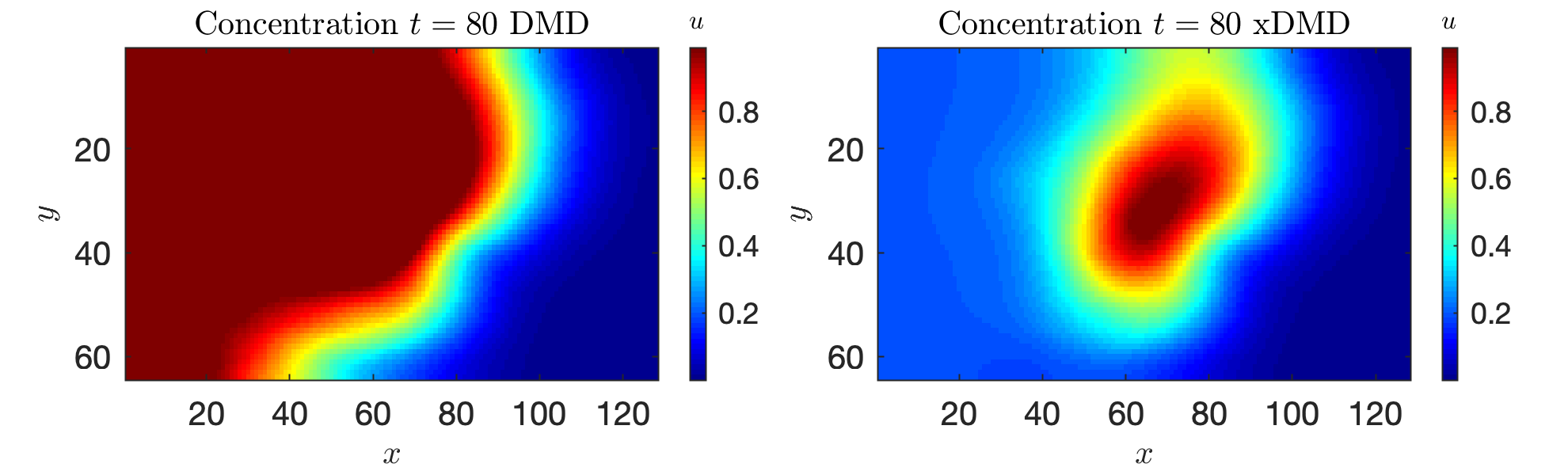}
\includegraphics{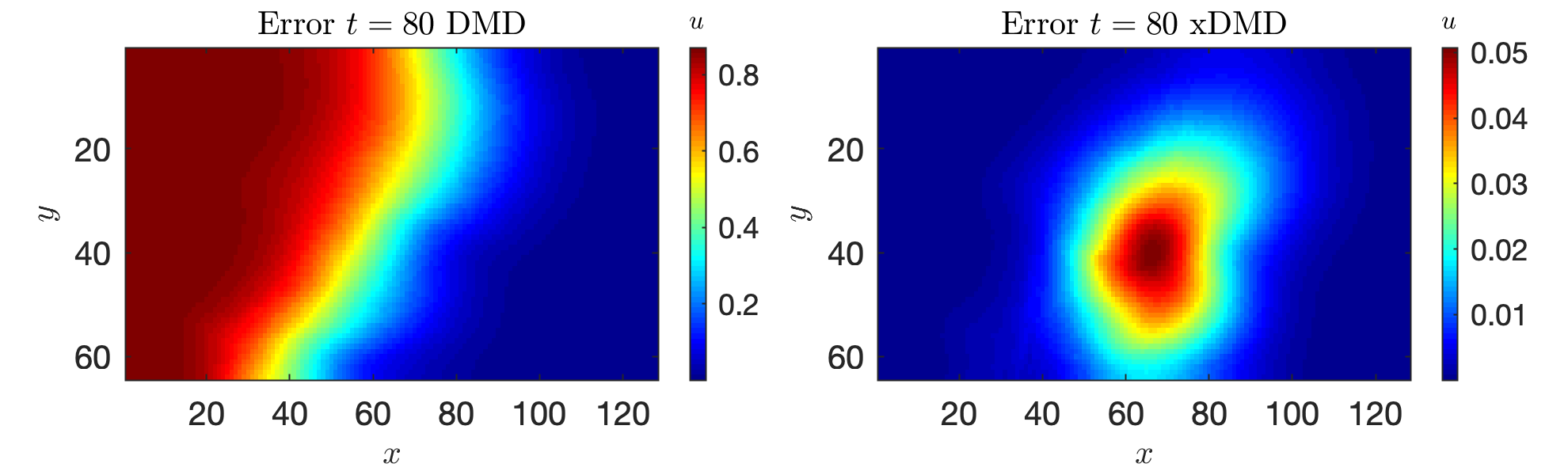}
\caption{Solute concentration predicted with the reference solution and the DMD and xDMD models for the initial condition not seen during training. Also shown are the absolute error of DMD and xDMD.}
\label{fig20}
\end{figure}

Our goal here is to test the ability of these models to predict $u(\bold x, T)$ for other initial conditions, such as the line source
\begin{equation}
u_\text{in} = \left\{\begin{aligned}
&80&x = 10, y\in[20,40],\\
&0&\mbox{otherwise.}
\end{aligned}\right.
\end{equation}

In Figure~\ref{fig20}, we compare the ability of DMD and xDMD to predict a quantity of interest, i.e., $u(\bold x,T)$, for an initial condition that is qualitatively different from that for which they were trained\footnote{The results for an initial condition given by a linear combination of two Gaussians are presented in Supplemental Material.}. While xDMD performs well in this generalizability test, DMD yields a wrong output concentration map because of its failure to handle inhomogeneity. The prediction error is largest in the vicinity of the left boundary, along which the inhomogeneous Dirichlet boundary condition is prescribed.

\section{Conclusions and Future Work}
\label{sec5}

We presented an extended DMD (xDMD) framework for representation of (linear or nonlinear) in- homogeneous PDEs. Our xDMD borrows from residual learning and bias identification ideas, which originated in the deep neural networks community. It shows high accuracy in learning the underlying dynamics, especially in inhomogeneous systems for which standard DMD fails. The inhomogeneous source can be accurately learned from the bias term at no extra computational cost. We conducted a number of numerical experiments to demonstrate that xDMD is an effective data-driven modeling tool and offers better accuracy than the standard DMD.

Although xDMD provides an optimal linear approximation of the unknown dynamics, data-driven modeling for highly nonlinear PDE in general remains a challenging task. Judiciously chosen observables are needed in order to approximate the corresponding Koopman operator, which requires either prior knowledge about the dynamics or dictionary learning. Developments and experiences from deep learning
may again bring potential solutions and vice versa.

In the follow-up work, we plan to use xDMD to construct surrogates, e.g., for Markov Chain Monte
Carlo solutions of inverse problems and for uncertainty quantifications. The verified generalizability will allow us to replace the expensive simulation with xDMD surrogates in each Monte Carlo run. Further
model reduction can be carried out to improve efficiency as well.

\section*{Acknowledgements}
This work was supported in part by by Air Force Office of Scientific Research under award numbers FA9550-17-1-0417 and FA9550-18-1-0474, and by the gift from Total.

\renewcommand\refname{Reference}
\bibliography{xDMD}

\newpage

\begin{center}
\textbf{SUPPLEMENTAL MATERIAL}
\end{center}
\setcounter{equation}{0}
\setcounter{figure}{0}
\setcounter{table}{0}
\makeatletter
\renewcommand{\theequation}{S\arabic{equation}}
\renewcommand{\thefigure}{S\arabic{figure}}

We provide a few additional test cases used to demonstrate the relative performance of DMD and xDMD.

\section*{Boundary Conditions and Noisy Data}

We study the non-homogeneity driven by boundary conditions. Consider the following one dimensional diffusion equation:
\begin{equation}
\partial_t u = D\partial_{xx}u, \ x\in [0,1], t>0, D = 0.1.
\end{equation}

Three different cases are tested to make comparison of DMD and gDMD:
\begin{itemize}
\item Case 1: Dirichlet boundary conditions,
\begin{equation}
\left\{
\begin{aligned}
&u(x,0) = 1,\\
&u(0,t) = 3, u(1,t) = 2.
\end{aligned}\right.
\end{equation}

\item Case 2: Neumann boundary conditions,
\begin{equation}
\left\{
\begin{aligned}
&u(x,0) = \exp(-20(x-0.5)^2),\\
&u_x(0,t) = 0, u_x(1,t) = 0.
\end{aligned}\right.
\end{equation}

\item Case 3: Contaminant training data,
We study the same initial and boundary conditions as in Case 1. The training data is now contaminant with with $0.1\%$ measurement noise.
\end{itemize}

The same spatial and temporal discretization as Test~\ref{test1} is used with the same number of training data. The solution behavior is trivial and thus omitted here. Accuracy is compared between DMD and xDMD in terms of representation, extrapolation and interpolation in Figure~\ref{fig3}.

In Case 1, we observe higher order of accuracy obtained by xDMD than DMD in all three tests of representation, extrapolation and interpolation. DMD can capture the overall solution behavior because the diffusion effect is dominant in the dynamic than the non-homogeneity driven by the boundaries. However, we observe that the DMD model error is mostly distributed near the two boundaries and the error accumulates with time. xDMD, on the other hand, has flat error distribution in the physical domain with much smaller error magnitude. 

In Case 2, which is a homogeneous case, we still observe higher order of accuracy obtained by xDMD than DMD. The improvements are mostly due to the modification in rDMD but also indicate that no sacrifice of accuracy is made by adding the bias. This test guarantees better performance of xDMD without knowledge of homogeneity.

In Case 3, both DMD and xDMD lose several orders of accuracy and behave almost the same in the presence of noise. In the interpolation test, xDMD is even less accurate than DMD. This can be explained by similar over-fitting issue in machine learning: the models with more parameters fit too closely to the limited number of contaminant data and therefore fail to fit additional data or predict future observations reliably. Obviously, xDMD has more parameters than DMD due to the bias term.

\begin{figure}[H]
\includegraphics{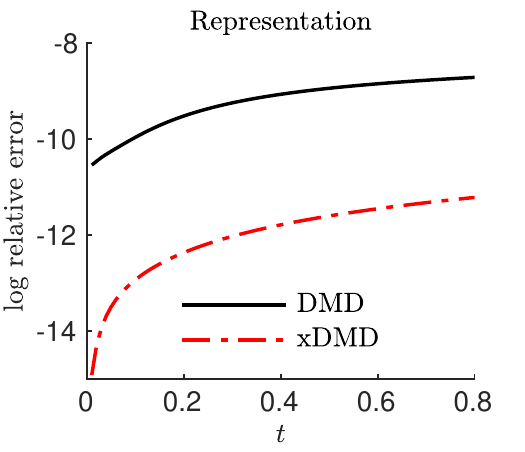}
\includegraphics{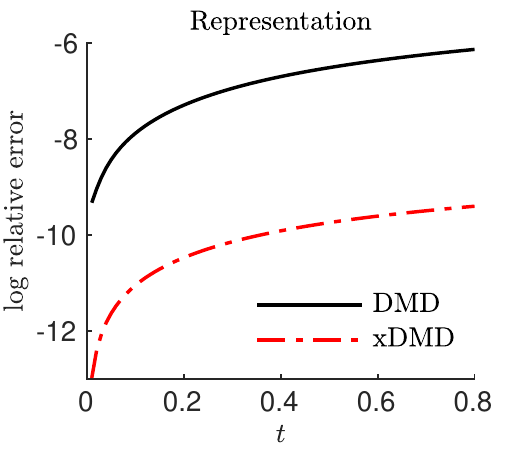}
\includegraphics{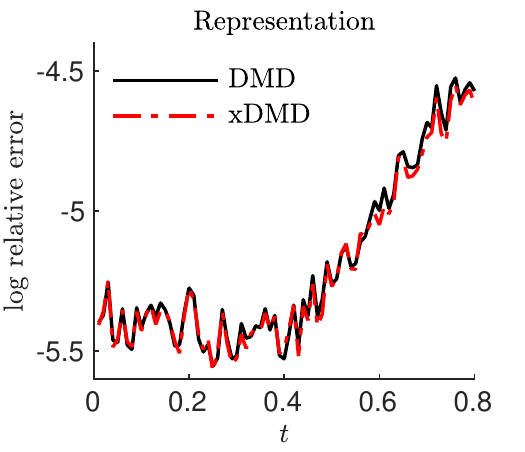}

\includegraphics{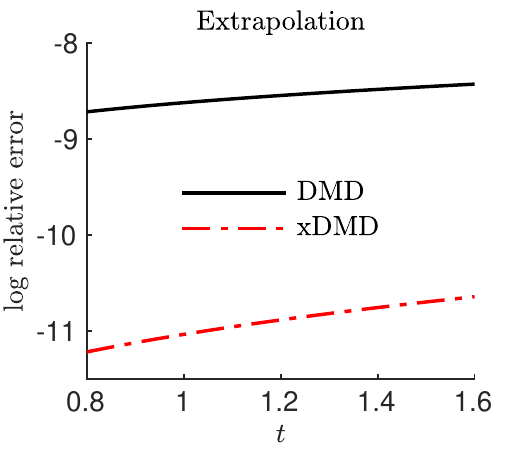}
\includegraphics{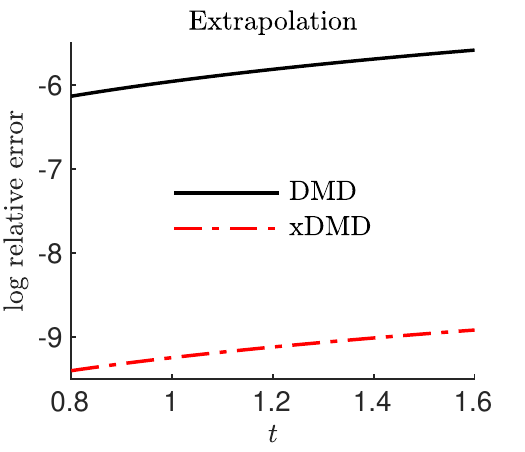}
\includegraphics{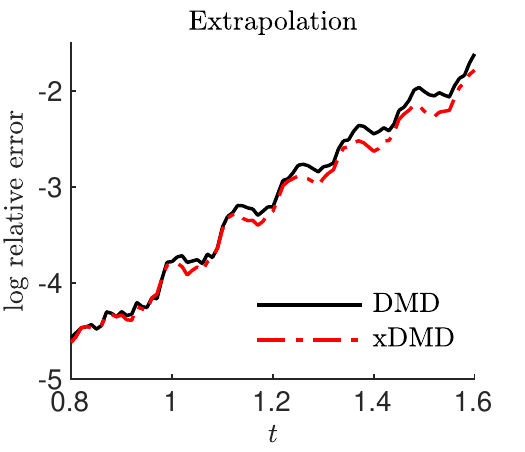}

\includegraphics{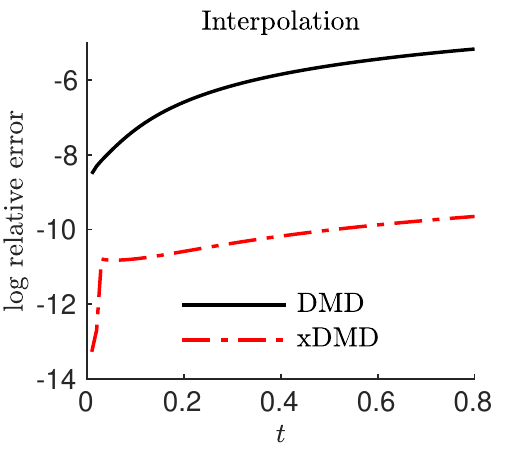}
\includegraphics{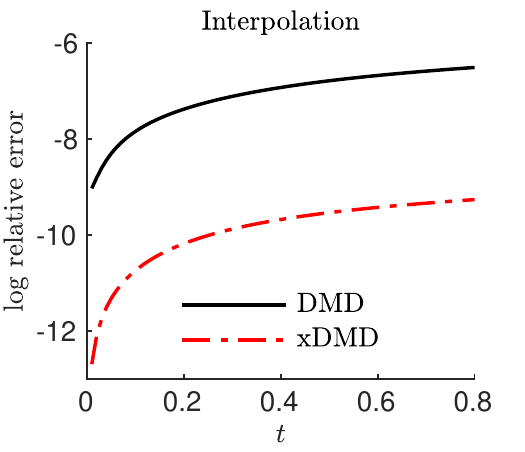}
\includegraphics{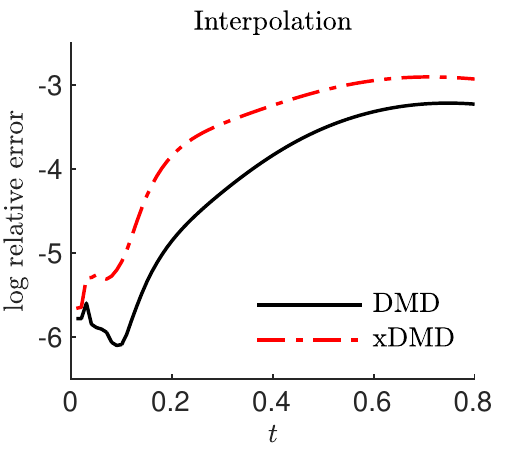}
\caption{\label{fig3}Accuracy tests of Case 1 to 3 in terms of representation, extrapolation and interpolation. Left Column: Case 1; Middle column: Case 2; Right column: Case 3. Top row: representation; middle row: extrapolation; bottom row: interpolation.}

\end{figure}


\section*{2D Viscous Burger's Equation}
\label{test4-4-1}

Consider the following two dimensional viscous Burger's equation with no flux boundary conditions:
\begin{equation}
\left\{
\begin{aligned}
&\partial_t u+ u\partial_x u+v\partial_y u = \nu (\partial_{xx}u+\partial_{yy}u),\\
&\partial_t v+ u\partial_x v+v\partial_y v = \nu (\partial_{xx}v+\partial_{yy}v),\\
&u(x,y,0) = \left\{
\begin{aligned}
1&&(x,y)\in [0.5,1]\times[0.5, 1],\\
0&&\mbox{otherwise}
\end{aligned}
\right.
\end{aligned}
\right. \ \ \nu = 0.05, \ (x,y)\in [0,2]\times[0,2], t\in [0,2],
\end{equation}

\begin{figure}[H]
\begin{center}
\includegraphics{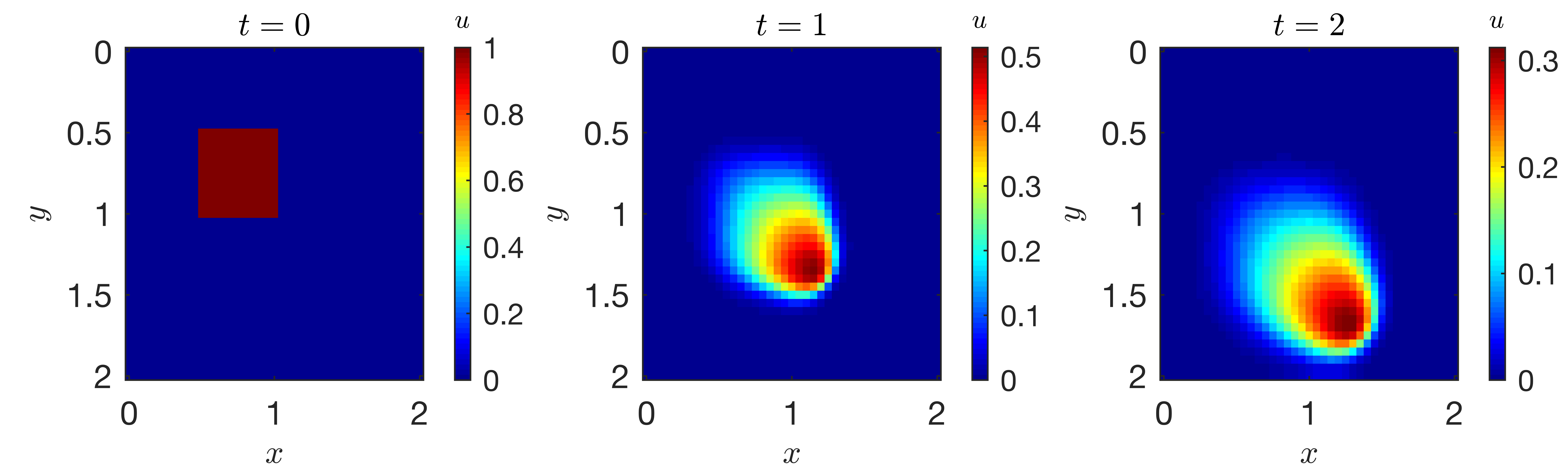}
\end{center}
\caption{Solution behavior of 2D viscous Burger's equation at different times.}
\label{fig7}
\end{figure}

\begin{figure}[H]
\begin{center}
\includegraphics{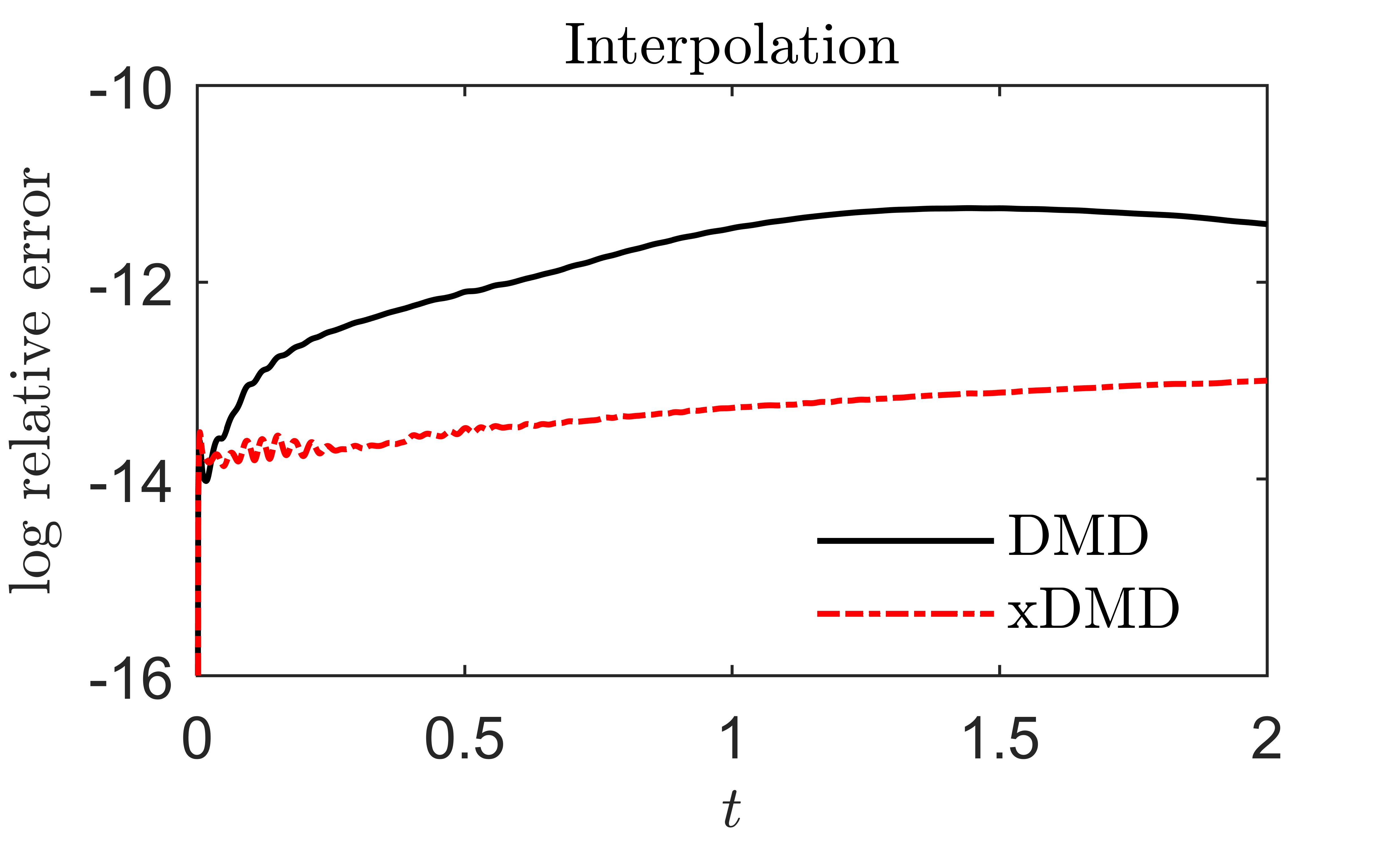}
\end{center}
\caption{Log relative error of DMD and xDMD vs. time in interpolation.}
\label{fig8}
\end{figure}
The reference solution is computed by a finite difference scheme using uniform mesh, where $\Delta x = \Delta y = 0.05$, $\Delta t = 0.001$. The snapshot solution needs to be reshaped into a vectorized form. We randomly select $500$ snapshots out of the $2000$ reference solutions to fill in the training data. Due to the viscosity, the solution presents weak nonlinearity and smooth diffusive profile as shown in Figure~\ref{fig7}. Both DMD and xDMD can capture the nonlinear flow with satisfactory accuracy. We plot the log relative error in Figure~\ref{fig8} and observe smaller error magnitude of xDMD as before.

\section*{1D Advection-Diffusion Equation}
\label{sec4-5-1}

First, we consider a one dimensional advection diffusion equation with a time-independent source:
\begin{equation}\label{eq4-12}
\left\{
\begin{aligned}
& \partial_t u+v\partial_x u=  D\partial_{xx}u+S(x), x\in [-4,4], t\in [0,4],  \ v = 1, \ D= 0.1,\\
&S(x) = \exp(-x^2/0.2).
\end{aligned}
\right.
 \end{equation}
The training will be conducted using the following initial and boundary conditions:
\begin{equation}
\left\{
\begin{aligned}
&u(x,0) = \exp(-(x+2)^2/0.1),\\
&u_x(-4,t) = 0, u_x(4,t) = 0.
\end{aligned}\right.
\end{equation}

The initial condition mimics a point source located at $x = -2$ with strength $1$ and correlation length $\sqrt{0.1}$. We address that the training data should be carefully chosen such that its traveling wave can cover the whole domain of interest and the training time should be long enough. For example in this case, one should choose a training dataset with active pulses all over the domain $[-4,4]$. Otherwise, the data-driven modeling will receive no signal in part of the domain and thus fail to learn the global dynamics. This issue has been discussed in~\cite{lu2020lagrangian} for advection dominant phenomena. 

The training data are collected from reference solutions using a finite difference scheme with $\Delta x = 0.04, \Delta t = 0.04$. As shown in Figure~\ref{fig13}, both DMD and xDMD can represent the training data with satisfactory accuracy. Same as previous tests, xDMD achieves higher order accuracy than DMD. 

\begin{figure}[H]
\includegraphics{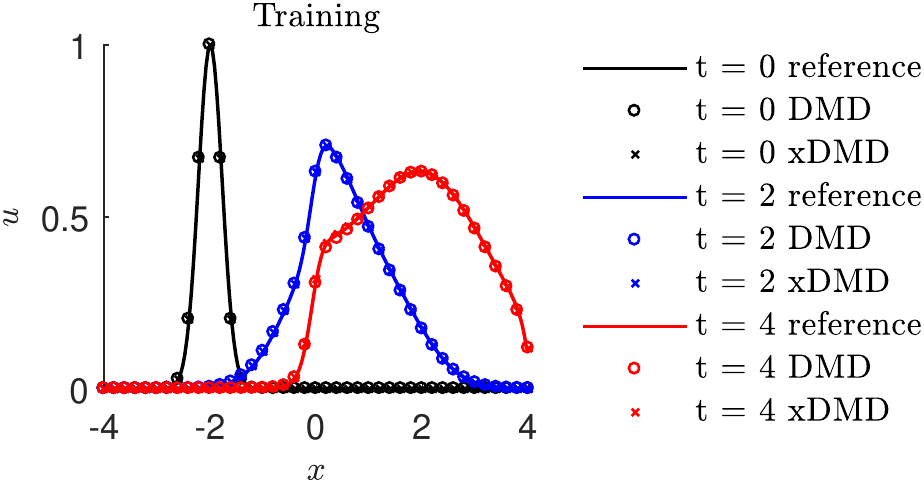}\hspace{1cm}
\includegraphics{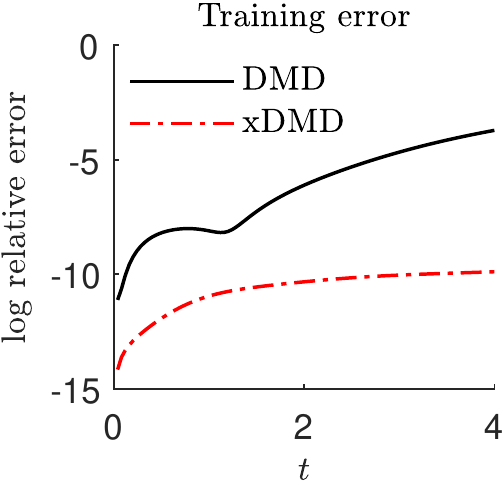}
\caption{DMD and xDMD in representing the training data. Left: modeling solutions compared with the reference solution at different times; Right: log relative error vs. time.}
\label{fig13}
\end{figure}

Essentially, we want the DMD and xDMD models to learn the advection diffusion operator with a fixed source from the training data. If the models are accurate, then for another different  initial inputs, e.g. a point source with different strength, location and correlation length, one can use the DMD and xDMD approximation to output solutions directly without simulating all over again. We use two different types of initial inputs to test the generalizability. In test 1, the initial data is generated from a single point source $u(x,0) = s\exp(-(x-x^0)^2/\sigma^2)$, where $s\sim \mathcal U[1,11], x^0\sim \mathcal U[-2,1], \sigma^2\sim \mathcal U[1/15,1/10]$. In test 2, the initial data is generated from a two-point source $u(x,0) = s_1\exp(-(x-x^0_1)^2/\sigma_1^2)+s_2\exp(-(x-x^0_2)^2/\sigma_2^2)$, where $s_1,s_2\sim \mathcal U[1,11], x^0_1, x^0_2\sim \mathcal U[-2,1], \sigma_1^2,\sigma_2^2\sim \mathcal U[1/15,1/10]$. 

Figure~\ref{fig14} shows that xDMD has superior performance in generalizing the learned model to new and previous unseen inputs. The modeling errors in the two tests are well controlled under reasonable magnitude. DMD, on the other hand, has poor performance in generalization due to the lack of source term identification. The nature of~\eqref{eq4-12} implies that a good model should consist of two parts: one part accounts for the advection-diffusion operator, which is sensitive to the variation of the initial inputs; the other part accounts for the inhomogeneous source term, which is invariant to the initial inputs. This intuition is well cooperated in the framework of xDMD.

\begin{figure}[H]
\includegraphics{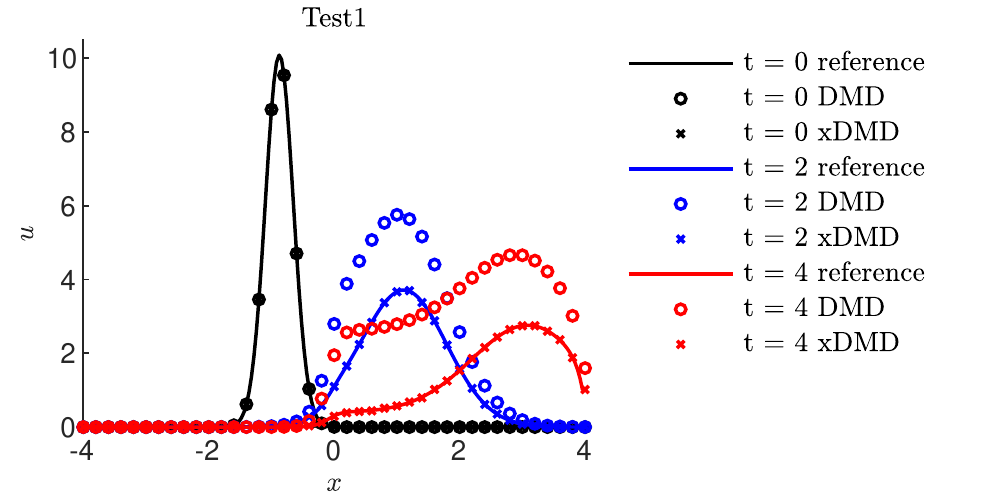}
\includegraphics{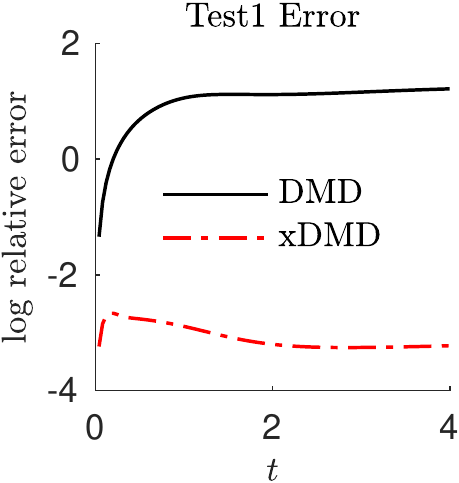}
\includegraphics{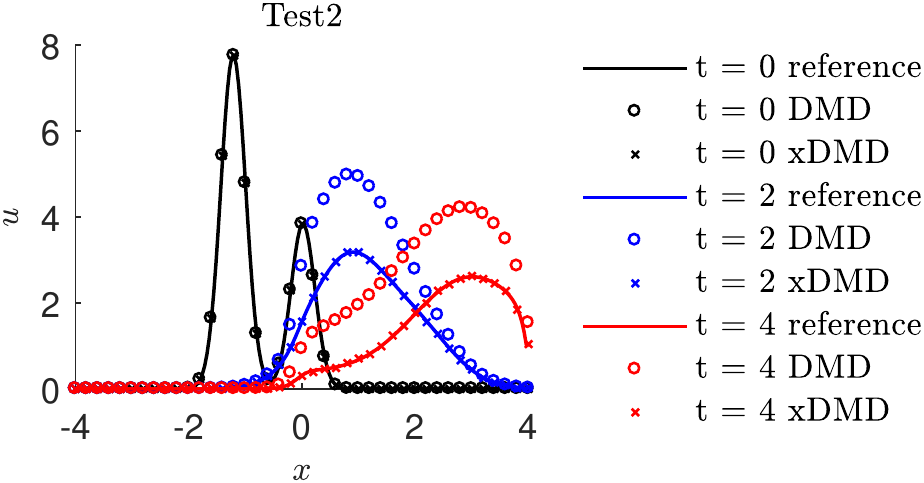}\hspace{1.9cm}
\includegraphics{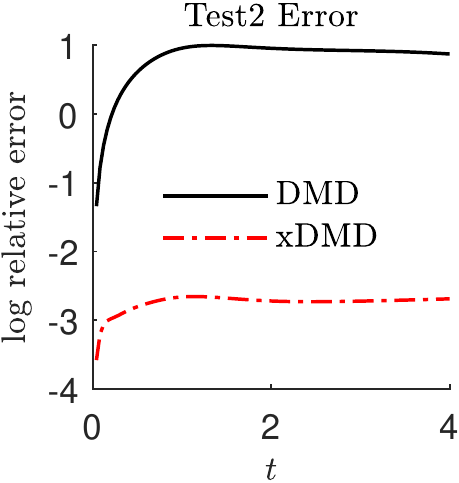}
\caption{DMD and xDMD in generalizability tests. Left: modeling solutions compared with the reference solution at different times; Right: log relative error vs. time.}
\label{fig14}
\end{figure}

\section*{2D Advection-Diffusion Equation}
\label{sec4-5-2}

Next, we consider a two-dimensional advection-diffusion equation with no flux boundary conditions:
\begin{equation}\label{eq4-14}
\left\{
\begin{aligned}
&\partial_t u+\mathbf v\cdot\nabla u=  \nabla\cdot(D\nabla u)+S(x,y), (x,y)\in [0,20]\times [0,10], t\in [0,T], \ \mathbf v = [-2.75,0], \ D= \begin{pmatrix}
0.5&0\\
0&0.5
\end{pmatrix}\\
&S(x,y) = s\exp(-\frac{1}{2\sigma^2}((x-5)^2+(y-5)^2)),\ s = 100, \ \sigma^2 = 0.05.
\end{aligned}
\right.
\end{equation}
This setting can be used to describe the release history of a contaminant in groundwater flows, where $u(x,y,t)$ stands for the concentration of the contaminant. Different from 1D case, it is impossible to select a proper initial inputs whose traveling wave can effectively cover the whole two dimensional domain. The generalizability will be tested in a different way as follows. The training data is done for the initial condition $u(x,y,0) = u_\text{in}(x,y)$ with 
\begin{equation}
u_\text{in} = s\exp\left[-\frac{1}{2\sigma^2}\left((x-x_s)^2+(y-y_s)^2\right)\right],
\end{equation}
where $s = 100, \sigma^2 = 0.05$ and the coordinates of the plume's center of mass, $(x_s,y_s)$ are treated as independent random variables with uniform distributions, $x_s\sim\mathcal U[0,10]$ and $y_s\sim\mathcal U[0,10]$. We generate $N_\text{MC}$ realizations of the pairs $(x_s,y_s)$ and evaluate the corresponding initial conditions $u_\text{in}^{(n)}(\bold x)$ for $n = 1,\dots, N_\text{MC}, \ N_\text{MC} = 4000$. For each of these realizations, \eqref{eq4-14} is solved to compute $u_T^{(n)}(\bold x) \equiv u^{(n)}(\bold x, T), \ T = 4$ using a finite difference scheme with $\Delta x = \Delta y = 0.25$. The matrix pairs $\{u_\text{in}^{(n)},u_T^{(n)}\}_{n=1}^{N_\text{MC}}$ are arranged into data matrix $\bold X$ and $\bold Y$ as in~\eqref{eq:4-10}. Finally, the DMD and xDMD models are deployed to learn the flow map $\boldsymbol \Phi_{\Delta t}$ with the time lag $\Delta t = T$. We evaluate the DMD and xDMD models in test data of the following three types:

\begin{itemize}
\item Test 1: Initial input is a single point source with different strength, location and correlation length:
\begin{equation}
u^0= s\exp\left[-\frac{1}{2\sigma^2}((x-x_{s})^2+(y-y_{s})^2))\right], s =\mathcal U(50,100), \sigma^2 = \mathcal U(0.02,0.1), x_{s}\sim \mathcal U[0,10],y_{s}\sim \mathcal U[0,10].
\end{equation}

\item Test 2: Initial input is a two-point source with different strength, location and correlation length:
\begin{equation}
\begin{aligned}
&u^0= s_1\exp\left[-\frac{1}{2\sigma_1^2}(x-x_{s_1})^2+(y-y_{s_1})^2)\right]+ s_2\exp\left[-\frac{1}{2\sigma_2^2}(x-x_{s_2})^2+(y-y_{s_2})^2)\right],\\
& s_{1},s_2 = \mathcal U(50,100)), \ \sigma_{1}^2,\sigma_2^2 = \mathcal U(0.02,0.1), \ x_{s_{1}},x_{s_2}\sim \mathcal U[0,10], \ y_{s_{1}},y_{s_2} \sim \mathcal U[0,10].
\end{aligned}
\end{equation}

\item Test 3: Initial input is a fixed strength line source:
\begin{equation}
u^0 = \left\{\begin{aligned}
&75&x = 5, y\in[3,6],\\
&0&\mbox{otherwise.}
\end{aligned}\right.
\end{equation}
\end{itemize}

Figure~\ref{fig15} shows the success of xDMD in learning the time lag $\Delta t = 4$ flow map with a totally different initial inputs than the training data. As long as the boundary conditions and source term are the same, the output concentration only depends on the input initial release. The xDMD modeling can recover the $\Delta t = 4$ flow map with high accuracy. However, DMD fails the generalization test due to the same reason as in section~\ref{sec4-5-1}. We notice that the error map of DMD has a peak centered at $(5, 5)$, which is the location of the source $S$ in equation~\eqref{eq4-14}. This further verifies that the loss of accuracy is caused by the shortcoming of DMD in identifying the inhomogeneous source term. 
\begin{figure}[H]
\includegraphics{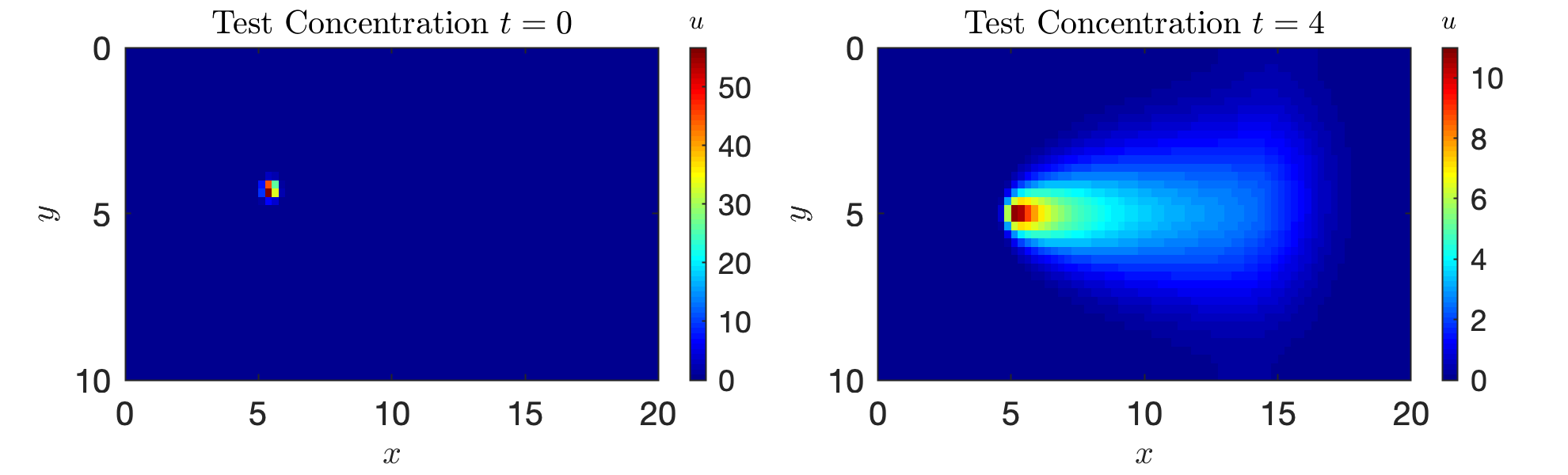}
\includegraphics{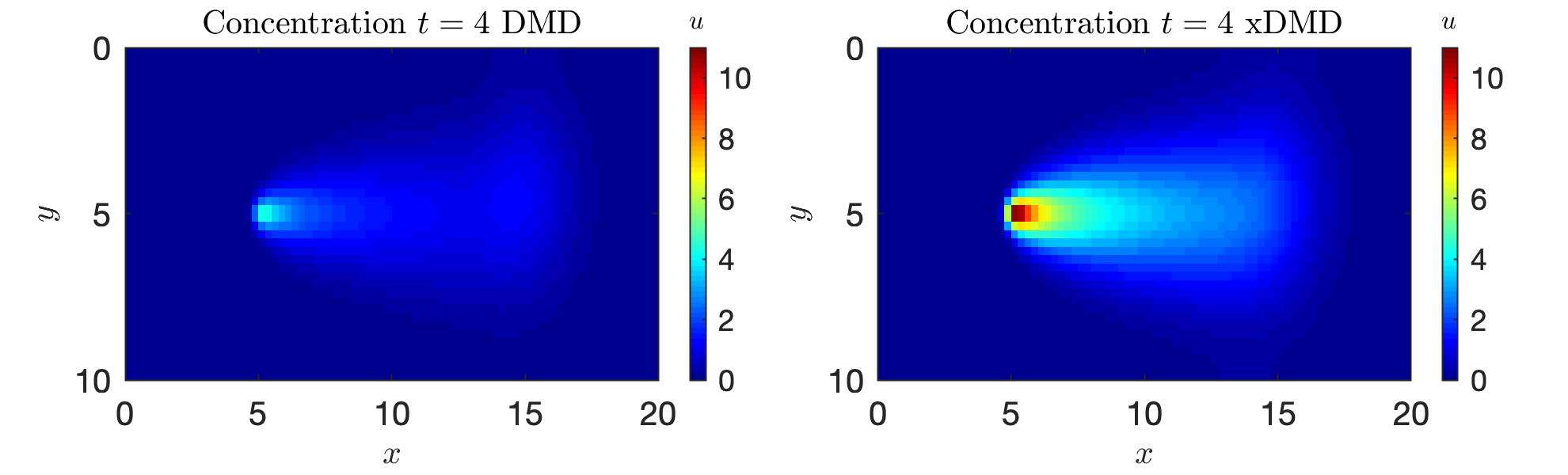}
\includegraphics{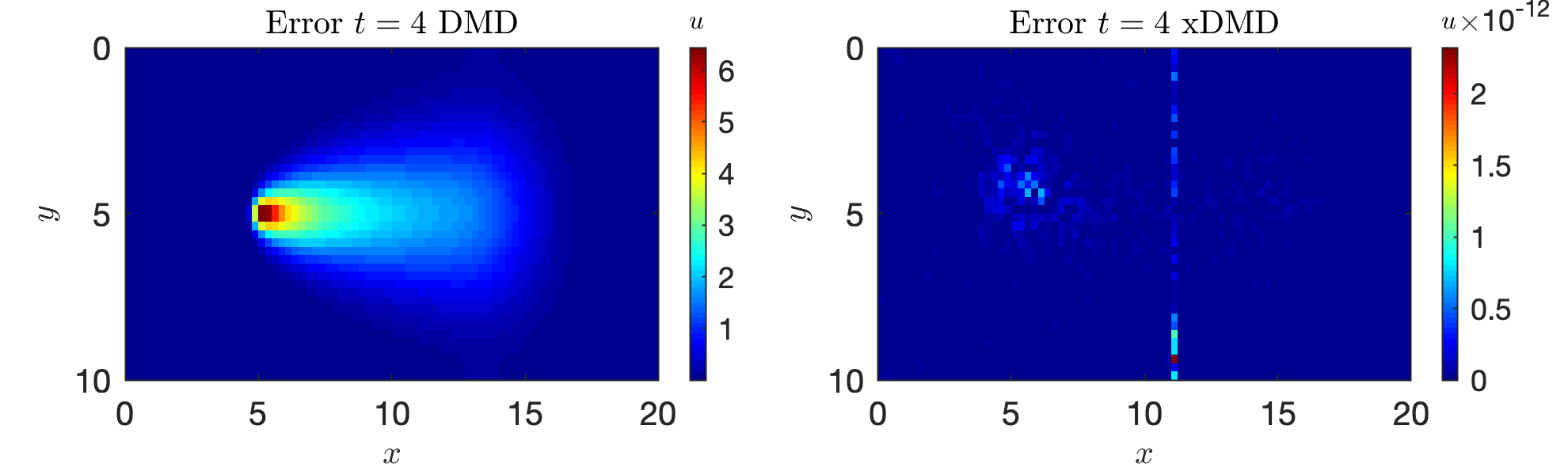}
\caption{Test 1: DMD and xDMD modeling solutions compared with the reference solution and absolute error maps.}
\label{fig15}
\end{figure}

Next in Figure~\ref{fig16}, we observe similar results of Test 2 as in Test 1. xDMD modeling is able to accurately output the concentration at $T = 4$ from the two-point source initial input. We notice that the right corner concentration tail is mostly caused by the advection-diffusion effect on the north-east point source. This pure advection-diffusion dynamic can be well captured by DMD as shown by the flat low error concentration in the DMD error map. The error peak is at $(5,5)$ again, showing the significant effect of identifying the source.

\begin{figure}[H]
\includegraphics{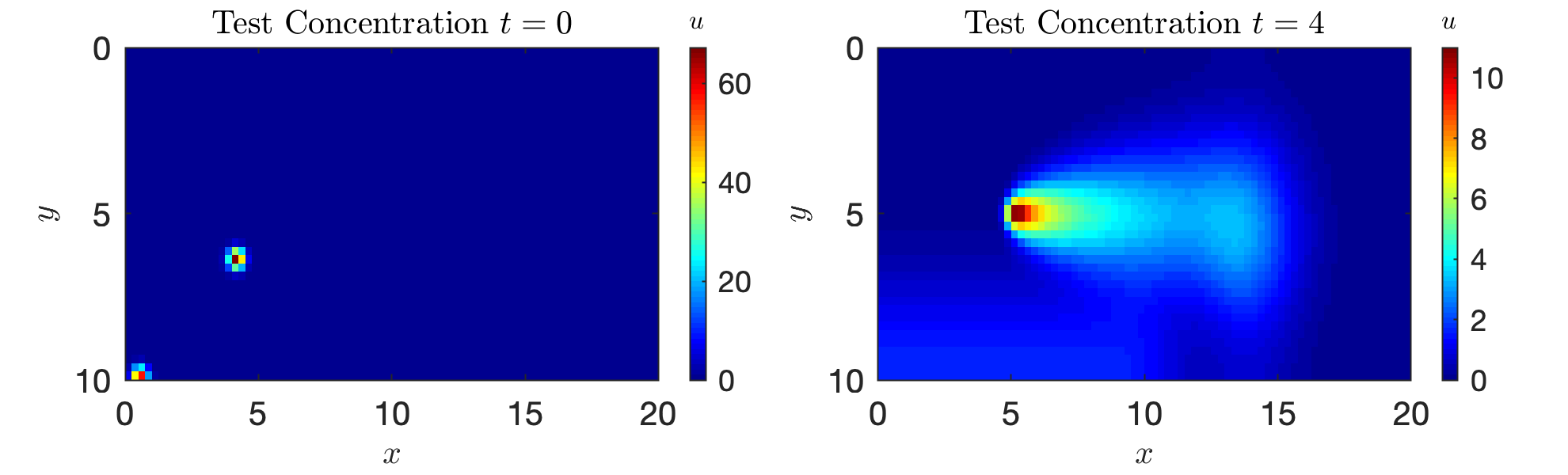}
\includegraphics{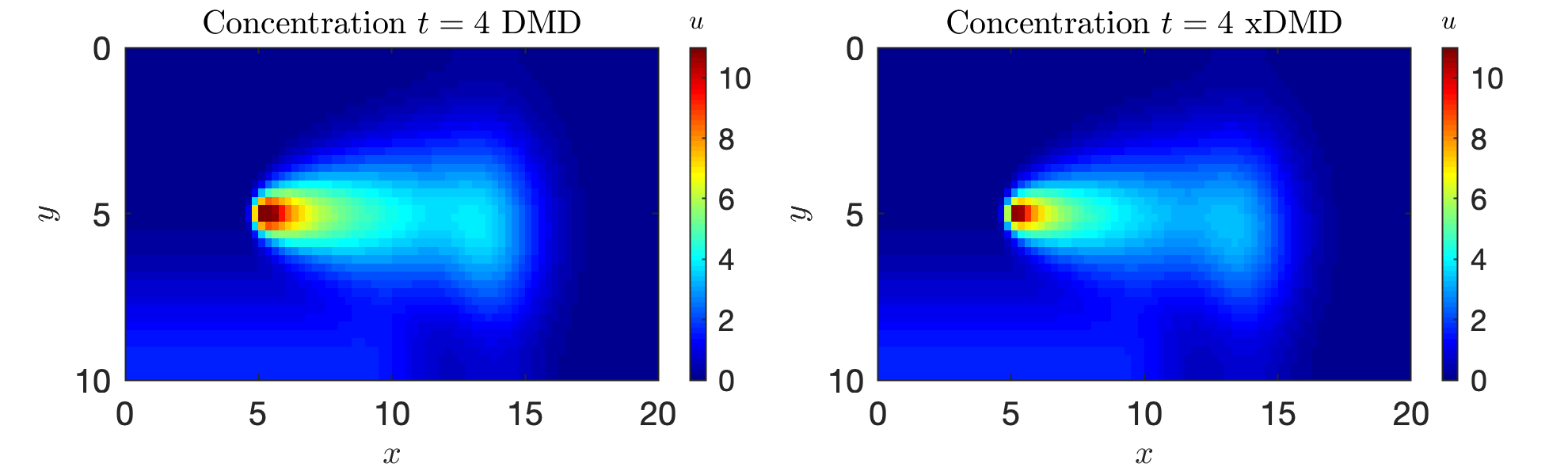}
\includegraphics{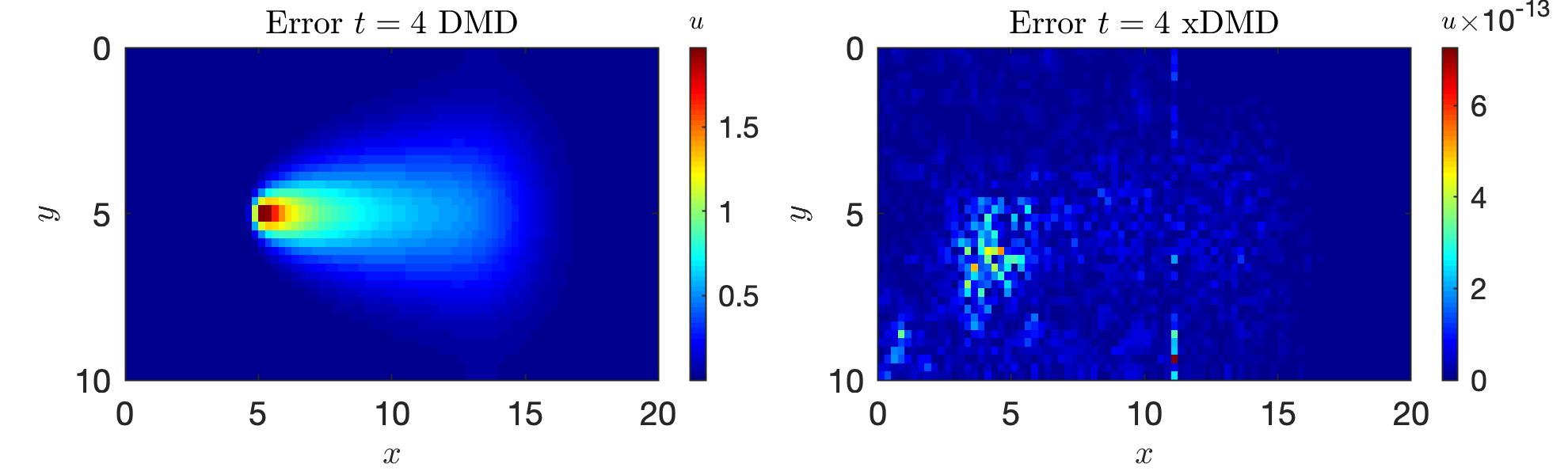}
\caption{Test 2: DMD and xDMD modeling solutions compared with the reference solution and absolute error maps.}
\label{fig16}
\end{figure}

Then in Figure~\ref{fig17}, we show the results in testing a line source initial input. Although the solution of single point source, two point source and line source present quite different features, all of them are essentially a linear superposition of the training single-point sources. Therefore, all of the three types inputs can be regarded as drawn from the same distribution. xDMD again achieves satisfactory accuracy in this generalizability test and DMD appears similar error map pattern centered at $(5,5)$ as before.

\begin{figure}[H]
\includegraphics{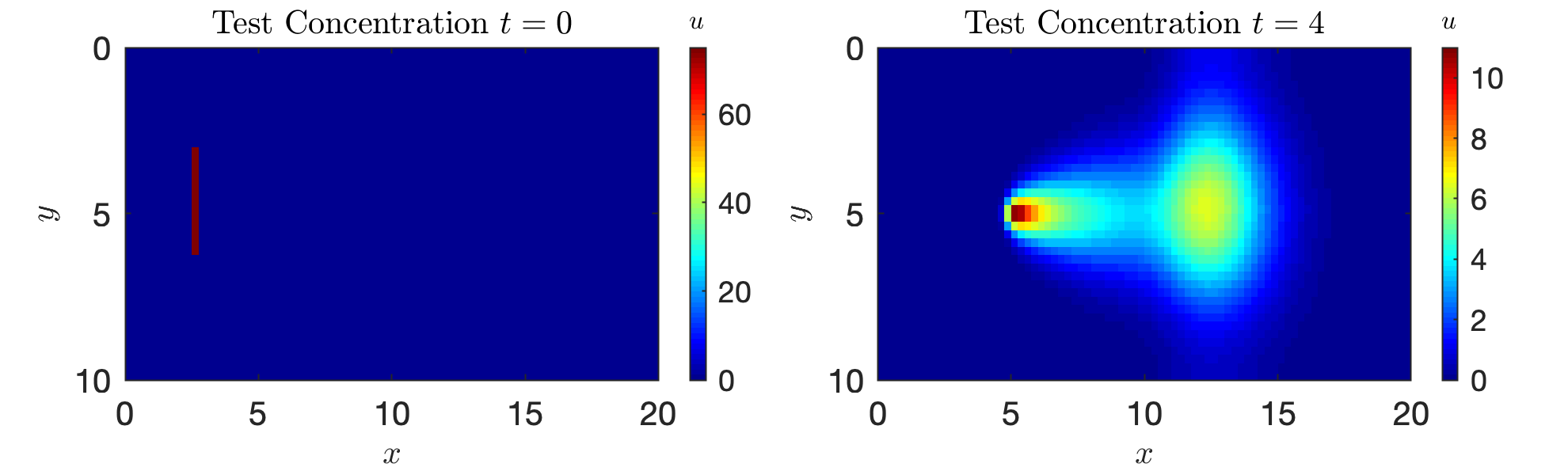}
\includegraphics{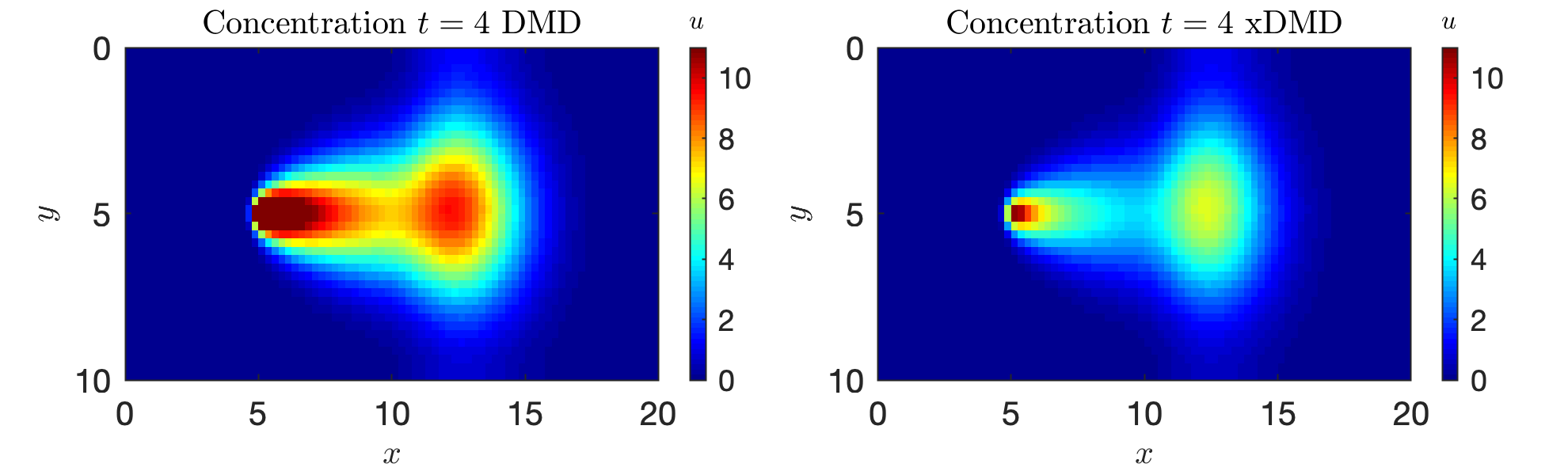}
\includegraphics{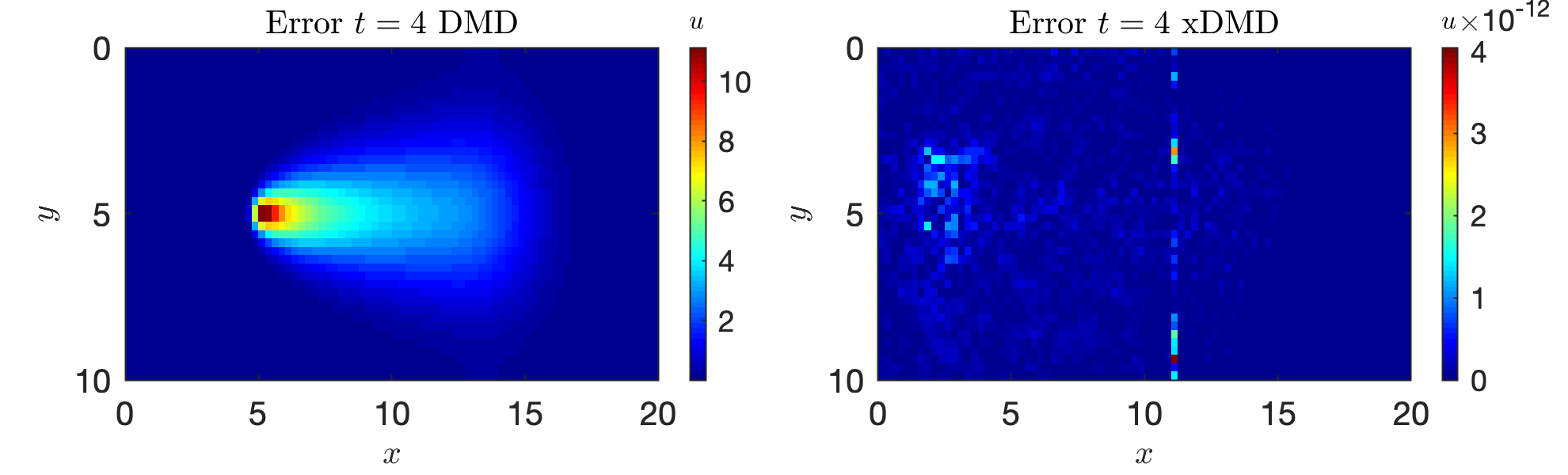}
\caption{Test 3: DMD and xDMD modeling solutions compared with the reference solution and absolute error maps.}
\label{fig17}
\end{figure}

\section*{Generalizability to New Inputs}
The setting is identical to that in Section~\ref{sec:generalizability}. Our goal here is to test the ability of these models to predict $u(\bold x,T)$ for other initial conditions, such as a two-point source with different strength and locations:

\begin{equation}
u_\text{in}(\bold x)= s_1\exp(-(x-x_{s_1})^2+(y-y_{s_1})^2)+ s_2\exp(-(x-x_{s_2})^2+(y-y_{s_2})^2,
\end{equation}
where $s_{1} = 50, \ s_2 = 80, \ (x_{s_{1}}, y_{s_{1}}) = (10,40),  \ (x_{s_{2}}, y_{s_2}) = (20,20)$. 
\begin{figure}[H]
\includegraphics{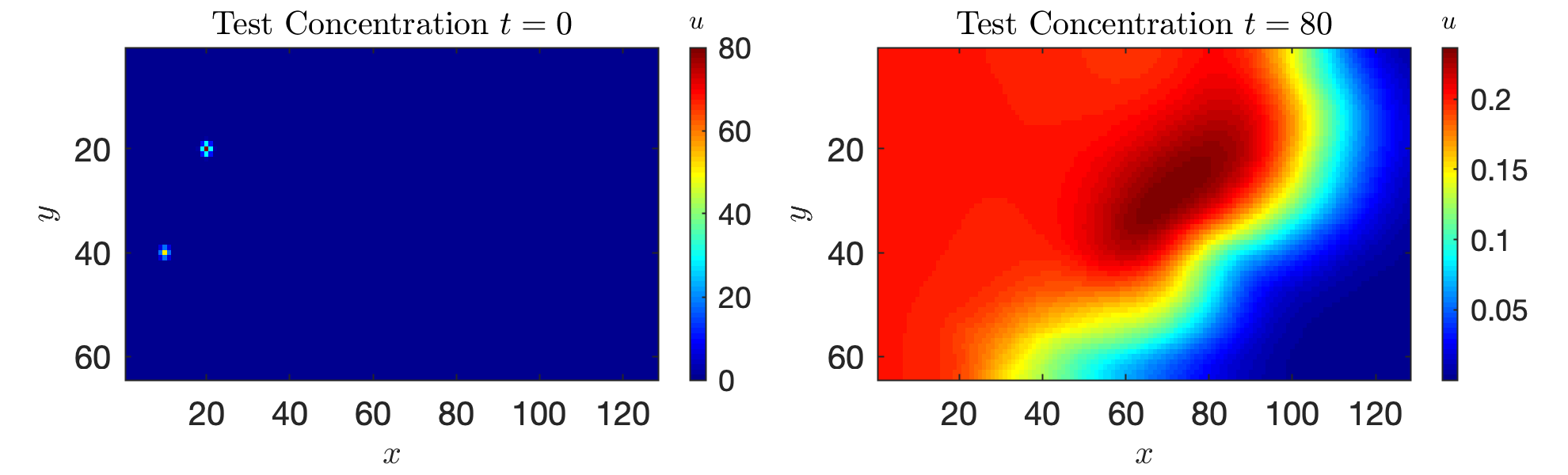}
\includegraphics{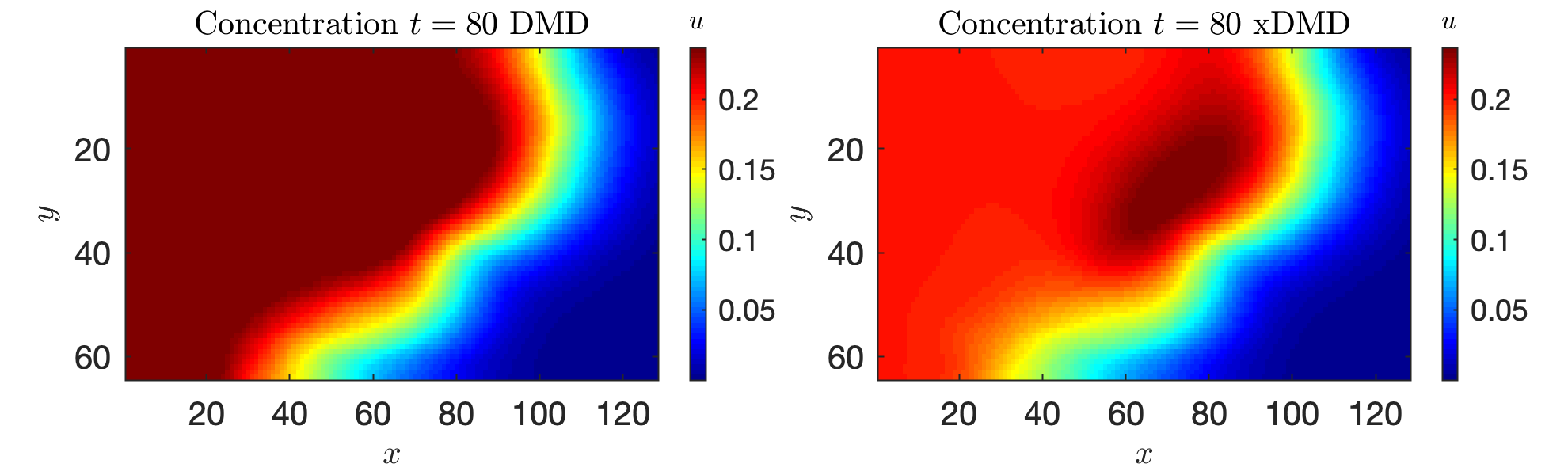}
\includegraphics{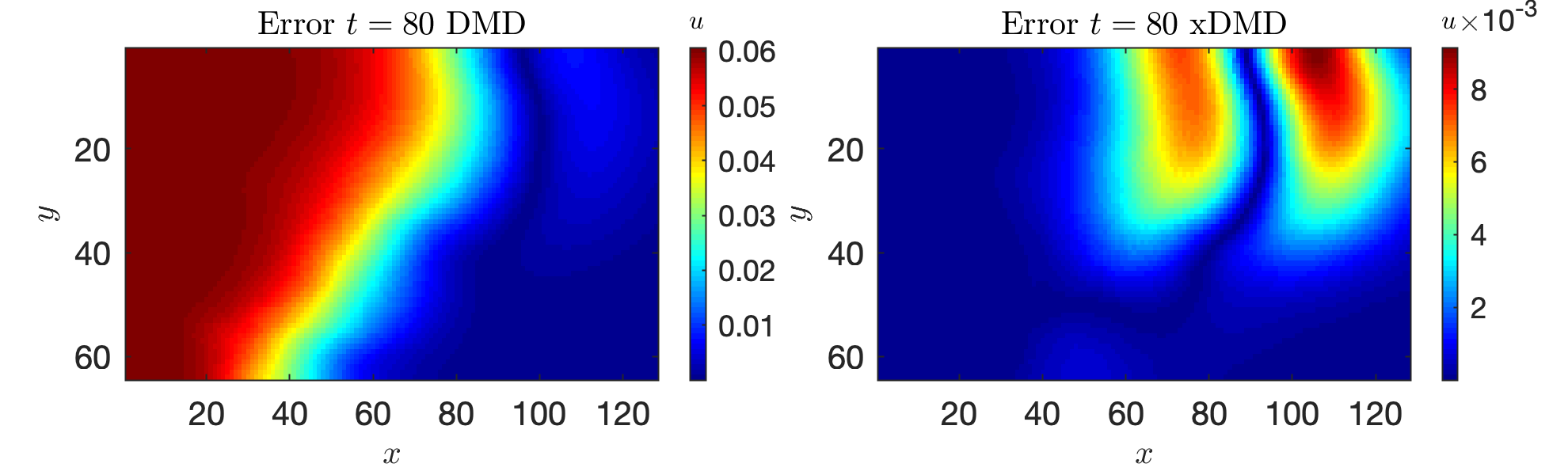}
\caption{DMD and xDMD modeling solutions compared with the reference solution and absolute error maps.}
\label{fig19}
\end{figure}

Figure~\ref{fig19} shows the success of xDMD in learning the time lag $\Delta t = 80$ flow map with a totally different initial inputs than the training data. The error map of xDMD presents very small magnitude and indicates the high accuracy of xDMD in this generalized test. On the other hand, DMD predicts a very different concentration map and fails the generalization test. Similarly as before, we observe that the error map of DMD arising from the left boundary, where Dirichlet boundary condition is imposed. This visualization again addresses the significant role of the bias term added in the new xDMD framework.

\end{document}